\documentclass[11pt]{amsart}
\usepackage{amssymb,mathrsfs,graphicx,extpfeil}
\usepackage{epsfig}
\usepackage{indentfirst, latexsym, amssymb, enumerate,amsmath,graphicx,stackengine}
\usepackage{float}
\usepackage{colortbl}
\usepackage{epsfig,subfigure}
\usepackage{amsmath}
\usepackage{dsfont}
\title[Collective dynamics for the stochastic kinetic Cucker-Smale equation]{Collective stochastic dynamics of the Cucker-Smale ensemble under uncertain communication}

\author[Ha]{Seung-Yeal Ha}
\address[Seung-Yeal Ha]{\newline Department of Mathematical Sciences and Research Institute of Mathematics \newline Seoul National University, Seoul 08826 \newline
and Korea Institute for Advanced Study, Hoegiro 87, Seoul 02455, Republic of Korea}
\email{syha@snu.ac.kr}

\author[Jung]{Jinwook Jung}
\address[Jinwook Jung]{\newline Department of Mathematical Sciences, \newline Seoul National University, Seoul, 08826, Republic of Korea}
\email{warp100@snu.ac.kr}

\author[R\"ockner]{Michael R\"ockner}
\address[Michael R\"ockner]{\newline Fakult\"at f\"ur Mathematik, \newline Universit\"at Bielefeld, Bielefeld 33615, Germany \newline
and Academy of Mathematics and System Sciences, Chinese Academy of Sciences, Beijing, China}
\email{roeckner@math.uni-bielefeld.de}

\newtheorem{theorem}{Theorem}[section]
\newtheorem{lemma}{Lemma}[section]
\newtheorem{corollary}{Corollary}[section]
\newtheorem{proposition}{Proposition}[section]

\newtheorem{remark}{Remark}[section]

\newtheorem{definition}{Definition}[section]

\newcommand{\bbr}{\mathbb R}

\newcommand{\bbe}{\mathbb E}

\newcommand{\bbn} {\mathbb N}

\newcommand{\e}{\varepsilon}

\def\charf {\mbox{{\text 1}\kern-.30em {\text l}}}

\begin{document}

\date{\today}

\subjclass{35Q82, 35Q92, 35R60} \keywords{Cucker-Smale model, emergence, flocking, random communication, stochastic kinetic Cucker-Smale equation}

\thanks{\textbf{Acknowledgment.} The work of S.-Y. Ha was supported by National Research Foundation of Korea(NRF-2017R1A2B2001864), and the work of J. Jung is supported by
the National Research Foundation of Korea (NRF) grant funded by the Korea government (MSIP) : NRF-2016K2A9A2A13003815. The work of M. R\"ockner is supported by the German Research Foundation (DFG) under the project number IRTG2235.}

\begin{abstract}
We study collective dynamics of the Cucker-Smale (C-S) ensemble under random communication. As an effective modeling  of the C-S ensemble with infinite size, we introduce a stochastic kinetic C-S equation with a multiplicative white noise. For the proposed stochastic kinetic model with a multiplicative noise, we present a global existence of strong solutions and their asymptotic flocking dynamics, when initial datum is sufficiently regular, and communication weight function has a positive lower bound.
\end{abstract}

\maketitle \centerline{\date}


\allowdisplaybreaks

\section{Introduction}
\setcounter{equation}{0}
Collective behaviors of self-propelled particles are ubiquitous in many biological systems in our nature, to name a few, flocking of birds, herding of sheep and swarming of fish, etc. Throughout the paper, we will use a terminology ``{\it flocking}" to denote aforementioned coherent collective motions. More precisely, {\it flocking} phenomenon denotes a situation in which self-propelled particles adjust their motions into a self-organized ordered motion using only the limited environmental information based on simple rules \cite{C-S, D-M, M-T1, M-T2, T-T, V-Z}. After Reynolds and Vicsek et al's pioneering works in \cite{Re, V-C-B-C-S}, several mechanical models were introduced in literature \cite{ C-F-T-V, C-H-L, C-D1, C-D2,  L-P-L-S, M-T1, P-L-S-G-P, P-E-G, Tad, V-Z} to model such coherent collective motions. Among others, our main interest lies on the mean-field kinetic model, namely ``kinetic C-S model \cite{H-K-Z, H-Liu, H-T}". Let $f := f(t,x,v)$ be the one-particle distribution function for the C-S ensemble at position $x$ with microscopic velocity $v$ at time $t$. Then, the dynamics of $f$ is governed by the kinetic C-S equation:
\begin{align}
\begin{aligned}\label{A-1}
&\partial_t f + v \cdot \nabla_x f + \nabla_v \cdot ({\tilde F}_a [f] f) = 0, \quad (t, x,v)\in \bbr_+ \times \bbr^{2d},\\
&F_a [f](t,x,v) := \int_{\bbr^{2d}} \phi(x_* - x) (v_* - v) f(t,x_*,v_*) dv_* dx_*.
\end{aligned}
\end{align}
Here $F_a[f]$ is a non-local operator measuring the attractive interactions between particles, and $\phi$ is a communication weight function which is nonnegative and radially symmetric:
\[\phi(x) = \bar{\phi}(|x|) \ge 0, \quad \forall x \in \bbr^d, \]
where $\bar{\phi} : [0,\infty) \to \bbr_+$ is Lipschitz continuous, bounded and monotonically decreasing: 
\begin{align*}
\begin{aligned}
& 0 \le \bar{\phi}(r) \le \bar{\phi}(0)=: \phi_M, \quad (\bar{\phi}(r)-\bar{\phi}(s))(r-s)\le 0,~~\forall~~r, s \in [0,\infty), \\
& \mbox{and}\quad [\phi]_{Lip}:= \sup_{r \neq s}\frac{|\bar{\phi}(r) - \bar{\phi}(s)|}{|r-s|}<\infty. 
\end{aligned}
\end{align*}
In the last decade, the mean-field equation \eqref{A-1} and its variants have been extensively studied from various perspectives, e.g., well-posedness and emergent dynamics \cite{C-F-R-T, H-T}, Fokker-Planck perturbation \cite{H-J-N-X-Z}, local sensitivity analysis \cite{H-J-J}, etc. For more detailed discussion, we refer to a recent survey article \cite{C-H-L}. In this paper, we are interested in the quantitative effects on the flocking dynamics of \eqref{A-1} due to the uncertain communication weights. Recently, the local sensitivity analysis for \eqref{A-1} has been discussed in an abstract and general framework in \cite{H-J-J} and some quantitative pathwise estimates for the variations in $f$ and its derivatives in random space were studied. However, authors  in \cite{H-J-J} could not provide interesting probabilistic estimates in relation with the emergent dynamics (see \cite{A-P-Z, C-P-Z, C-M, H-J, L-W} for a related local sensitivity analysis in uncertainty quantification (UQ)).  Thus, our primary goal of this paper is to address some probability estimate for \eqref{A-1} with uncertain communication. 

To fix the idea, we incorporate a stochastic noise into the communication weight, i.e. $\phi \to \phi + \sigma \circ \dot{W}_t$, where ${\dot W}_t$ is a one-dimensional white noise on the probability space $(\Omega, \mathcal{F}, \mathbb{P})$, $\sigma$ denotes the strength of the noise and $\circ$ denotes the stochastic integral in Stratonovich's sense. Then formally, under the unit mass assumption $\int_{\bbr^{2d}} f(t,x,v) dx dv = 1,$ the non-local operator $F_a[f]$ is replaced by a combination of the deterministic part $F_a[f]$ and stochastic part involving with ${\dot W}_t$:
\begin{equation} \label{A-1-1}
F_a[f] \quad \Longrightarrow \quad F_a[f] + \sigma (v_c-v) \circ \dot{W}_t. 
\end{equation}
Now, we combine \eqref{A-1} and \eqref{A-1-1} to derive the stochastic kinetic C-S equation:
\begin{equation} \label{A-2}
\partial_t f_t + v \cdot \nabla_x f_t + \nabla_v \cdot (F_a[f_t]f_t)  = \sigma \nabla_v \cdot ((v-v_c)f_t) \circ \dot{W}_t.
\end{equation}
For a rigorous derivation, we refer to Section 2.1 and here, we use the standard notation for random probability density function $f_t(x,v) := f(t, x, v)$. \newline

At the particle level, the effects of white noise perturbations were discussed in \cite{A-H, E-H-S, H-L-L}. Moreover, a rigorous derivation of equation \eqref{A-2} as a mean-field limit of the C-S systems with multiplicative noises was recently discussed in \cite{C-S2} based on the propagation of chaos in \cite{C-F}, and a mean-field limit of the C-S systems with another type of stochastic perturbations was also addressed in \cite{Ro}.  However, as far as the authors know, the kinetic C-S equation \eqref{A-1-1} perturbed by a multiplicative white noise has only been addressed in measure spaces such as $\mathcal{P}_2$, not in other function spaces (e.g. Sobolev spaces). For other types of stochastic kinetic equations, we refer to \cite{F-G-P,P-S}. In this paper, we address the following two questions: \newline
 \begin{itemize}
\item
(Well-posedness):~Is the stochastic kinetic C-S equation \eqref{A-2} well posed in a suitable function space such as Sobolev spaces? 

\vspace{0.2cm}

\item
(Emergence of flocking): ~If so, does the solution to \eqref{A-2} exhibit asymptotic flocking dynamics?
\end{itemize}

\vspace{0.2cm}

Our main results in this paper provide affirmative answers to the above posed questions. First, we introduce a concept of a strong solution to \eqref{A-2} and then provide a global well-posedness for strong solutions by employing a suitable regularization method and stopping time argument. Second, we provide a stochastic flocking estimate by showing that the expectation of the second velocity moment decays to zero exponentially fast, when the communication weight function $\phi$ has a positive infimum $\phi_m := \inf_{x \in \bbr^d} \phi(x)$ and noise strength $\sigma$ is sufficiently small compared to $\phi_m$. The main difficulty in our analysis arises, when we prove the existence of a solution to the regularized equation. Here, we obtain $W^{m,\infty}$-estimates for the sequence of functions that approximates the regularized equation. Our $W^{m,\infty}$-estimates contain terms with infinite expectation. Hence, even though we can find a limit function of the sequence from the pathwise estimates, it is not certain that the limit function becomes a solution to the regularized equation (see Remark 4.4 for detailed discussion). To cope with this problem, we used stopping time argument to get a solution to the regularized equation. \newline

The rest of this paper is organized as follows. In Section 2, we provide a rigorous derivation of equation \eqref{A-2} from the C-S system with a multiplicative noise, and then briefly discuss our main results on the global well-posedness for strong solutions and asymptotic flocking estimates of classical solutions. In Section 3, we provide several a priori estimates for classical solutions to \eqref{A-2}. In Section 4, we show our global well-posedness and emergent dynamics for strong solutions to \eqref{A-2}. Finally, Section 5 is devoted to a brief summary of our main results and discussion on future works.\\

\noindent {\bf Gallery of notation} Throughout this paper, we denote $(\Omega, \mathcal{F}, \mathbb{P})$ by a generic probability space. For $m,k \in \bbn$ and $p\in[1,\infty]$, we write $W^{m,p}(\bbr^k)$ as $m$-th order $L^p$ Sobolev space on $\bbr^k$. For $(x,v) \in \bbr^{2d}$, $\delta_{(x,v)}$ denotes a point mass concentrated at $(x,v)$.\\

For each $p\in[1,\infty)$, we denote $\mathcal{P}_p(\bbr^{2d})$ by
\[\mathcal{P}_p(\bbr^{2d}) := \{ \mu \ : \ \mbox{probability measure on $\bbr^{2d}$ such that } \int_{\bbr^{2d}} |(x,v)|^p d\mu(x,v) <\infty\}, \]
and we write $p$-Wasserstein distance on $\mathcal{P}_p(\bbr^{2d})$ as
\[W_p(\mu, \nu) := \left(\inf_{\gamma\in\prod(\mu,\nu)} \int_{\bbr^{4d}} |(x,v)-(y,w)|^p d\gamma \right)^{1/p},  \]
where $\prod(\mu,\nu)$ denotes the collection of all measures on $\bbr^{4d}$ whose marginals are $\mu$ and $\nu$.

\vspace{0.5cm}

\section{Preliminaries}\label{sec:2}
\setcounter{equation}{0}
In this section, we provide a rigorous derivation of the equation \eqref{A-2} and present our main results on the global existence of strong solutions to \eqref{A-2} and emergent flocking dynamics.

\subsection{Derivation of the SPDE}
In this subsection, following \cite{C-S2}, we present a derivation of \eqref{A-2} from the C-S system perturbed by a multiplicative noise. To be specific, we begin our discussion with the C-S model \cite{C-S}. 

Let $(x_t^i, v_t^i) \in \bbr^d \times \bbr^d$ be the position and velocity of the $i$-th particle at time $t\ge0$, respectively. Then, the ensemble of C-S particles is governed by the following system: 
\begin{align}
\begin{aligned}\label{C-S}
&dx_t^i = v_t^i dt, \quad t >0, \quad 1 \le i \le N,\\
&dv_t^i = F_a[\mu_t^N](x_t^i, v_t^i)dt, \quad \mu_t^N := \frac{1}{N}\sum_{i=1}^N \delta_{(x_t^i, v_t^i)},
\end{aligned}
\end{align}
where the flocking force $F_a$ is given in $\eqref{A-1}_2$. However, in a real world situation, the communications among particles are subject to the neighboring environment, e.g. drag forces from the fluid, abrupt disconnection, etc, which can be regarded as an intrinsic randomness in the model. To reflect these effects in the communication, stochastic noises can be incorporated into the communication weight $\phi$ appearing in system \eqref{C-S}. To address a stochastic perturbation in system \eqref{C-S}, we replace $\phi$ by $\phi + \sigma \circ \dot{W}_t$  and yield the following system of stochastic differential equations:
\begin{align}
\begin{aligned}\label{SC-S}
&dx_t^i = v_t^i dt, \quad t > 0,~~1 \leq i \leq N, \\
&dv_t^i = F_a[\mu_t^N](x_t^i, v_t^i)dt + \sigma ({\bar v}_t -v_t^i)\circ dW_t^i, \quad \bar{v}_t := \frac{1}{N}\sum_{i=1}^N v_t^i.
\end{aligned}
\end{align}

Let us compare \eqref{SC-S} with the model presented in \cite{A-H}, where the authors replaced  $\phi$ in \eqref{C-S} by $\phi + \sigma \dot{W}_t$ to obtain the C-S system with a multiplicative noise in It\^o's sense. However, we adopt the integral in Stratonovich's sense rather than It\^o's sense, since it enables us to use the method of stochastic characteristics once we derive a stochastic partial differential equation from system \eqref{SC-S}. Moreover, it is natural in the following sense: for each $1\le i \le N$, let $W_t^{i,\e}$ be a smooth approximation to the Wiener process $W_t^i$ (e.g. approximation by using a mollifier). Now, we consider the following system of deterministic equations: 
\begin{align}
\begin{aligned}\label{SC-S_reg}
&dx^{i,\e}_t = v^{i,\e}_t dt, \quad t > 0,~~1 \leq i \leq N, \\
&dv^{i,\e}_t = F_a[\mu^{N,\e}_t](x^{i,\e}_t, v^{i,\e}_t)dt+ \sigma \left({\bar v^\e}_t -v^{i,\e}_t\right)dW^{i,\e}_t,
\end{aligned}
\end{align}
Then, the Wong-Zakai theorem \cite{S-V, W-Z, W-Z2} implies that the solution to system \eqref{SC-S_reg} converges in probability to the solution to system \eqref{SC-S}. Here, we note that system \eqref{SC-S} is equivalent to the following It\^o equation \cite{Ev}:

\begin{align}
\begin{aligned}\label{SC-S2}
&dx_t^i = v_t^i dt, \\
&dv_t^i = \left[F_a[\mu_t^N](x_t^i, v_t^i)-\frac{1}{2}\sigma^2 ({\bar v}_t - v_t^i)\right]dt + \sigma ({\bar v}_t -v_t^i) dW_t^i.
\end{aligned}
\end{align}
When $W^i$'s are i.i.d Wiener processes, a similar analysis as in \cite{H-J-N-X-Z} yields the mean field limit of system \eqref{SC-S2} as $N\to\infty$, which is the following Fokker-Planck type equation:
\begin{equation}\label{CSFP}
\partial_t f + v \cdot \nabla_x f + \nabla_v \cdot \left[\left(F_a[f] - \frac{1}{2}\sigma^2 (v_c - v)\right)f\right] = \sigma\Delta_v(|v-v_c|^2 f), \quad x,v\in\bbr^d, \quad t>0,
\end{equation}
where $v_c := \int_{\bbr^d \times \bbr^d} vf dxdv$. However, if each $W_t^i$ is identical to a single Wiener process $W$, i.e. $W^i \equiv W_t$, we can use a propagation of chaos result \cite{C-S2} to obtain that the empirical measure $\mu_t^N$ associated with system \eqref{SC-S} converges to a measure-valued solution to \eqref{A-2}. Let us summarize the results on the mean-field limit and asymptotic flocking estimates in \cite{C-S2} as follows.
\begin{theorem}\label{T2.1}
Suppose that $T>0$ and consider a communication weight $\phi$ with $\bar\phi \in \mathcal{C}_b^1(\bbr_+)$, and let $\mu_0, \tilde{\mu}_0 \in \mathcal{P}_2(\bbr^d \times \bbr^d)$ be compactly supported in velocity. Then, the following assertions hold.
\begin{enumerate}
\item
 If $\mu$ and $\tilde{\mu}$ are two measure-valued solutions to \eqref{A-2} with compactly supported initial data $\mu_0$ and $\tilde{\mu}_0$ in velocity,  then
\[W_2(\mu_t, \tilde{\mu}_t) \le CW_2(\mu_0, \tilde{\mu}_0)e^{C(1+W_2(\mu_0, \tilde{\mu}_0)}, \quad \mbox{for a.s. $t \in [0, T]$}, \]
where the constant $C$ depends only on $\phi$, $T$, $\sigma$, $\sup_{t \in [0,T]}|B_t|$, and the support in velocity of $\mu_0$ and $\tilde{\mu}_0$. 

\vspace{0.2cm}

\item
If $\mu_0^N := \frac{1}{N}\sum_{i=1}^N \delta_{(x_0^i, v_0^i)}$ is an initial atomic measure such that 
\[ W_2(\mu_0, \mu_0^N) \to 0 \quad \mbox{as $N \to \infty$}, \]
then the empirical measure $\mu_t^N := \frac{1}{N}\sum_{i=1}^N \delta_{(x_t^i, v_t^i)}$ associated with system \eqref{SC-S} is a measure-valued solution to \eqref{A-2} with initial data $\mu_0^N$. Moreover, it converges almost surely to the measure-valued solution $\mu_t$ corresponding to the initial measure $\mu_0$:
\[\sup_{0\le t \le T}W_2(\mu_t, \mu_t^N) \le  CW_2(\mu_0, \mu_0^N)e^{C(1+W_2(\mu_0, \mu_0^N))} \to 0, \quad \mbox{as } \ N \to \infty, \]
\end{enumerate}
\end{theorem}

Note that the stability estimate in Wasserstein metric implies the uniqueness of measure-valued solutions in $\mathcal{P}_2(\bbr^d \times \bbr^d)$. 

\begin{theorem}\label{T2.2}
Suppose that the communication weight function $\phi$ satisfies 
\[ 0<\phi_m \le \phi(x) \le \phi_M \quad \mbox{for $x\in\bbr^d$}, \]
 and let $\mu_t$ be a measure-valued solution to \eqref{A-2}. Then we have
\[\bbe[E_0]e^{-2(\psi_M - \sigma^2)t} \le \bbe[E_t] \le \bbe[E_0]e^{-2(\psi_m - \sigma^2)t}, \]
where $E_t$ is defined as
\[E_t := \int_{\bbr^d \times \bbr^d} |\bar{v}_0 - v|^2 \mu_t(dx, dv), \quad \bar{v}_0 := \int_{\bbr^d \times \bbr^d} v\mu_0(dx,dv). \]
\end{theorem}
\begin{remark}
The results in \cite{C-S2} imply that the equation \eqref{A-2} can be derived as a mean-field limit of the particle system \eqref{SC-S}. Now, our objective is to establish solutions with higher regularity than measure-valued solutions.
\end{remark}

\vspace{0.4cm}

\subsection{Presentation of main results}
In this subsection, we provide our main results on the global well-posedness of \eqref{A-2} and emergent flocking dynamics. First, we consider the Cauchy problem for \eqref{A-2}:
\begin{align}
\begin{aligned} \label{B-0}
& \partial_t f_t + v \cdot \nabla_x f_t + \nabla_v \cdot (F_a[f_t]f_t)  = \sigma \nabla_v \cdot ((v-v_c)f_t) \circ \dot{W}_t, \quad  (t, x,v)\in \bbr_+ \times \bbr^{2d}, \\
& f_0(x,v) = f^{in}(x,v),
\end{aligned}
\end{align}
where the initial datum $f^{in}$ is assumed to be deterministic.  Next, we provide a definition for a strong solution to the Cauchy problem \eqref{B-0} as follows. 
\begin{definition}\label{D2.1}
For a given $T \in (0,\infty]$, $f_t = f_t(x,v)$ is a strong solution to \eqref{B-0} on $[0,T]$ if it satisfies the following relations:
\begin{enumerate}
\item
(Regularity):~For $k \ge 1$, $f_t \in \mathcal{C}([0,T]; W^{k,\infty}(\bbr^{2d}))$ a.s. $\omega \in \Omega$.

\vspace{0.2cm}

\item
(Integral relation):~$f_t$ satisfies the equation \eqref{A-2} in distribution sense: for $\psi \in {\mathcal C}^{\infty}_c([0,T] \times\bbr^{2d})$, 
\begin{align}
\begin{aligned} \label{B-0-1}
\int_{\bbr^{2d}} f_t \psi \ dvdx &= \int_{\bbr^{2d}} f^{in} \psi \ dvdx + \int_0^t \int_{\bbr^{2d}} f_s \left( v \cdot \nabla_x \psi + F_a[f_s] \cdot \nabla_v \psi \right)\ dvdx ds \\
&\hspace{0.2cm} - \sigma \int_0^t \Big( \int_{\bbr^{2d}} [(v-v_c) f_s] \cdot \nabla_v \psi \  dvdx \Big) \circ dW_s, \quad \mbox{a.s.~ $\omega \in \Omega$}.
\end{aligned}
\end{align}
\end{enumerate}
\end{definition}

\begin{remark}
1. We say $f_t$ is a classical solution to \eqref{B-0} if it is a $\mathcal{F}_t$-semimartingale satisfying relation \eqref{B-0} pointwise and the regularity condition $f_t \in L^\infty(\Omega; \mathcal{C}([0,T];\mathcal{C}^{3,\delta}(\bbr^{2d})))$ for some $\delta \in (0,1)$. We require this regularity condition to use It\^o's formula and the relation between It\^o and Stratonovich integration without any restriction.

\noindent 2. As can be seen later, the representation of a classical solution to \eqref{B-0} via the stochastic characteristics shows that $f_t$ can not satisfy the $L^\infty$-boundedness over $\Omega$ due to the exponential Wiener process. To handle this, we would use a suitable stopping time.
\end{remark}

Next, we are ready to provide a framework $(\mathcal{F})$ and main results below: 
\begin{itemize}
\item
($\mathcal{F}1$):~The initial datum $f^{in}$ is nonnegative, compactly supported in $x$ and $v$ and independent of $\omega$.

\vspace{0.2cm}

\item
($\mathcal{F}2$):~For $k \geq 1$, $f^{in}$ and $\phi$ are assumed to be in $W^{k,\infty}(\bbr^{2d})$ and $\mathcal{C}^{\infty}(\bbr^{2d})$, respectively.

\vspace{0.2cm}

\item
($\mathcal{F}3$):~The first two moments of $f^{in}$ are normalized as follows:
\[
\int_{\bbr^{2d}} f^{in}dv dx = 1, \quad \int_{\bbr^{2d}} vf^{in}dvdx = 0.
\]
\end{itemize}

Under the framework $({\mathcal F})$, our main results can be summarized as follows.
\begin{theorem}\label{T2.3}
Let $T \in (0, \infty)$ and assume that $f^{in}$ and $\phi$ satisfies the framework $(\mathcal{F})$. Then, there exists a strong solution $f_t$ to \eqref{B-0} on $[0,T]$ such that
\[ \bbe\|f_t\|_{L^\infty} \le \|f^{in}\|_{L^\infty}\exp\left\{\left(d\phi_M + \frac{(\sigma d)^2}{2}\right) t \right\}, \quad \mathbb{E}[M_2](t)  \leq M_2(0)\exp(2\sigma^2 t), \quad t \in [0, T). \]
Moreover, if a strong solution $f_t$ exists on $(0,\infty)$ and $\phi_m := \inf_{x \in \bbr^N} \phi(x) >\sigma^2 $, then one obtains an asymptotic flocking estimate:
\[ \mathbb{E}[M_2](t) \le M_2(0) \exp(-2(\phi_m - \sigma^2)t), \quad t >0. \]
\end{theorem}
\begin{proof}
For a proof, we first regularize the initial datum using  the standard mollification and then solve the linearized system for \eqref{B-0} to get a sequence of approximate solutions. Then, we use the stopping time argument to get a strong solution for \eqref{B-0} with the given initial datum. The detailed proof  will be presented in Section \ref{sec:4}.
\end{proof}

\begin{remark}
Note that for $k > 3$, a strong solution $f_t$ to \eqref{B-0} can be shown to satisfy the equation \eqref{B-0} pointwise within our framework.
\end{remark}

\section{A priori estimates for classical solutions}\label{sec:3}
\setcounter{equation}{0}
In this section, we study a priori estimates for classical solutions to \eqref{A-2}. First, we study several equivalent relations to the weak formulation \eqref{B-0-1}, when a strong solution satisfies suitable conditions.

\begin{lemma}\label{L3.1}
Suppose that for every $\psi \in \mathcal{C}_c^\infty(\bbr^{2d})$ and a random process $f_t \in L^\infty(\Omega \times [0,T] \times \bbr^{2d})$, $\int_{\bbr^{2d}} f_t \psi dvdx $ has a continuous $\mathcal{F}_t$-adapted modification, where $\{\mathcal{F}_t\}$ is a family of $\sigma$-field generated by the Wiener process. Then, $f_t$ is a $\mathcal{F}_t$-semimartingale satisfying relation \eqref{B-0-1} if and only if for every $\psi \in \mathcal{C}_c^\infty(\bbr^{2d})$, 
\begin{align}
\begin{aligned} \label{B-0-2}
\int_{\bbr^{2d}} f_t \psi \ dvdx &= \int_{\bbr^{2d}} f^{in} \psi \ dvdx + \int_0^t \int_{\bbr^{2d}} f_s \left( v \cdot \nabla_x \psi + F_a[f_s] \cdot \nabla_v \psi \right)\ dvdx ds \\
&\hspace{0.2cm} - \sigma \int_0^t \Big( \int_{\bbr^{2d}} [(v-v_c) f_s] \cdot \nabla_v \psi \  dvdx \Big) dW_s\\
& \hspace{0.2cm} +\frac{\sigma^2}{2} \int_0^t \int_{\bbr^{2d}} (v-v_c) f_s \cdot \Big[ \nabla_v \big( (v-v_c) \cdot \nabla_v \psi \big) \Big] dvdxds  \quad \mbox{a.s.~ $\omega \in \Omega$}.
\end{aligned}
\end{align}
\end{lemma}
\begin{proof}
The proof is almost the same as in Lemma 13 from \cite{F-G-P}, but we provide a sketch for a proof for readers' convenience.  Note that the following relation between It\^o and Stratonovich integrals holds:
\[\int_0^t h_s \circ dW_s = \int_0^t h_s dW_s + \frac{1}{2}\langle h, W \rangle_t, \]
where $\langle \cdot, \cdot \rangle$ denotes the joint quadratic variation (see \cite{H.K}).  In our case, $h_s$ corresponds to $\int_{\bbr^{2d}}[(v-v_c) f_s] \cdot \nabla_v \psi dvdx $. Then, to deal with $\langle h, W \rangle_t,$, one needs to specify the stochastic part of $h_s$. Here, if we replace $\psi$ in \eqref{C-2} by $(v-v_c) \cdot \nabla_v \psi$, we can find out that the stochastic part of  $ h_s$ becomes $-\sigma\int_0^t \Big[  \int_{\bbr^{2d}} ((v-v_c)  f_s) \cdot \nabla_v (v \cdot \nabla_v \psi)dvdx \Big]dW_s$. This means
\[\Big\langle\int_{\bbr^{2d}}[(v-v_c) f_{\Large \cdot}] \cdot \nabla_v \psi dvdx , W \Big\rangle_t = -\sigma\int_0^t \int_{\bbr^{2d}} [(v-v_c)f_s] \cdot \nabla_v [(v-v_c) \cdot \nabla_v \psi]dvdx ds, \]

and we may conclude the proof here.
\end{proof}

Once we reformulate relation \eqref{B-0-1} to It\^o form \eqref{B-0-2},  we can show that the solution process $f_t$ satisfies the following pointwise relation under the regularity condition for $f$.

\begin{lemma}\label{L2.2}
Suppose that $f_t \in L^\infty(\Omega; \mathcal{C}([0,T]; \mathcal{C}^2(\bbr^{2d})))$ has a continuous $\mathcal{F}_t$-adapted modification and has a compact support in $x$ and $v$. Then, $f_t$ satisfies relation \eqref{B-0-2} if and only if $f_t$ satisfies the following relation:
\begin{align}
\begin{aligned}\label{B-0-3}
 f_t &(x,v) \\
&= f^{in}(x,v)  - \int_0^t \Big(v \cdot \nabla_x f_s  + \nabla_v \cdot (F_a[f_s]f_s) \Big) ds  + \sigma \int_0^t \Big[\nabla_v \cdot \big( (v-v_c) f_s \big) \Big] dW_s\\
& \quad + \frac{\sigma^2}{2} \int_0^t  \nabla_v \cdot \Big[(v-v_c) \nabla_v \cdot \big((v-v_c) f_s\big) \Big] ds, \quad \mathbb{P}  \otimes dx \otimes dv \mbox{-a.s.}
\end{aligned}
\end{align}
\end{lemma}

\begin{proof}
First, we assume that $f$ satisfies \eqref{B-0-2}. Since $f_t$ is smooth and compactly supported, we use Fubini's theorem to show that \eqref{B-0-2} is equivalent to

\begin{align}
\begin{aligned}\label{B-0-4}
\int_{\bbr^{2d}} f_t \psi \ dvdx &= \int_{\bbr^{2d}} f^{in} \psi \ dvdx - \int_0^t \int_{\bbr^{2d}}\Big[v \cdot \nabla_x f_s + \nabla_v \cdot (F_a[f_s]f_s) \Big] \psi \ dvdx ds \\
&\hspace{0.2cm}  + \sigma \int_0^t \Big( \int_{\bbr^{2d}} \nabla_v \cdot [(v-v_c) f_s]  \psi \  dvdx \Big) dW_s\\
& \hspace{0.2cm} +\frac{\sigma^2}{2} \int_0^t \int_{\bbr^{2d}} \nabla_v \cdot \Big[(v-v_c) \nabla_v \cdot \big((v-v_c) f_s\big) \Big]  dvdxds  \quad \mbox{a.s.~ $\omega \in \Omega$}.
\end{aligned}
\end{align}
Note that for each $\psi \in \mathcal{D}(\bbr^{2d})$, it satisfies the relation \eqref{B-0-4}  outside $\mathbb{P}$-zero set depending on the choice of $\psi$. We recall from standard functional analysis that $\mathcal{D}( \bbr^{2d})$ is separable, i.e. there exists $\{\psi_i \}_{i=1}^\infty \subseteq \mathcal{D}(\bbr^{2d})$ which is dense in $\mathcal{D}(\bbr^{2d})$. Here, we choose $\Omega_i \subset \Omega$ such that $\mathbb{P}(\Omega_i)=1$ and \eqref{B-0-1} holds for $f_t$ and $\psi_i$ over $\Omega_i$. Let $\tilde\Omega := \cap_{i=1}^\infty \Omega_i$. Then $\mathbb{P}(\tilde\Omega)=1$ and \eqref{B-0-1} holds for any $\psi_i$ and $f_t$ over $\tilde\Omega$.\newline

\noindent Now, we show $f_t$ satisfies the relation \eqref{B-0-3}. For this, we define functionals $\mathscr{L}_t[f]$ and $\mathscr{M}[f_t]$ as follows:
\begin{align*}
\begin{aligned}
&\mathscr{L} [f_t](x,v) := f_t - f^{in} + \int_0^t   \Big(v \cdot \nabla_x f_s  + \nabla_v \cdot (F_a[f_s]f_s) \Big) ds \\
& \hspace{2.5cm} - \frac{\sigma^2}{2} \int_0^t  \nabla_v \cdot \Big[(v-v_c) \nabla_v \cdot \big((v-v_c) f_s\big) \Big] ds,\\
&\mathscr{M} [f_t]  := \nabla_v \cdot [(v-v_c) f_t].
\end{aligned}
\end{align*}
For a given $(x^*,v^*) \in \bbr^{2d}$, we can choose a sequence $\{\rho_i\} \subset \mathcal{D}(\bbr^{2d})$, using the standard mollifier technique or other tools,  such that for any $i \in \bbn$,

\[ \Big| (\rho_i * \mathscr{L}[f_t])(x^*, v^*)  - \mathscr{L}[f_t](x^*, v^*) \Big| + \int_0^t \Big| (\rho_i * \mathscr{M}[f_s])(x^*, v^*)  - \mathscr{M}[f_s](x^*, v^*) \Big|^2 ds \le \frac{1}{2^{i+1}}, \]
 where the regularity and compact support of $f$ can be used to guarantee the above inequality. We also use the denseness of $\{\psi_i\}$ to obtain $\{\tilde\psi_i\} \subseteq \{\psi_i\}$ which satisfies, for any $i \in \bbn$,
\[  \Big| (\rho_i - \tilde{\psi}_i)*\mathscr{L}_t[f](x^*, v^*) \Big| +  \int_0^t \Big| (\rho_i-\tilde\psi_i)* \mathscr{M}[f_s](x^*, v^*) \Big|^2 ds \le \frac{1}{2^{i+1}}. \]
Thus, we have
\begin{equation}\label{B-0-5}
\Big(\tilde\psi_i * \mathscr{L}_t [f]\Big) (x^*, v^*) \longrightarrow \mathscr{L}_t [f](x^*, v^*).
\end{equation}
Moreover, we use It\^o isometry to get
\begin{equation}\label{B-0-6}
\mathbb{E} \left[  \Big(\int_0^t (\tilde\psi_i * \mathscr{M}[f_s] - \mathscr{M}[f_s]) dW_s \Big)^2 \right] = \mathbb{E} \left[ \int_0^t (\tilde\psi_i * \mathscr{M}[f_s] - \mathscr{M}[f_s])^2 ds \right] \longrightarrow 0.
\end{equation}
Hence, we can obtain the convergence of \eqref{B-0-4} with $\psi = \tilde\psi_i(x^*-x, v^*-v)$  towards \eqref{B-0-3} at $(x^*, v^*)$ as $i \to \infty$, by combining \eqref{B-0-5} and \eqref{B-0-6}. We perform this procedure to obtain that for every $(x^*,v^*) \in \bbr^{2d}$, $f$ satisfies relation \eqref{B-0-3} $\mathbb{P}$-a.s. and this gives
\[ \mathbb{E} \left[\left|\mathcal{L}[f_t] - \int_0^t \mathscr{M}[f_s] dW_s \right|(x,v)\right] = 0, \]
for every $(x,v) \in \bbr^{2d}$. Thus, we use Fubini theorem to get 
\[ \mathbb{E}\left[\int_{\bbr^{2d}}\left|\mathcal{L}[f_t] - \int_0^t \mathscr{M}[f_s] dW_s\right| dvdx \right] = 0. \]
This implies our first assertion. \newline

\noindent Next, we assume that $f$ satisfies \eqref{B-0-3} $\mathbb{P} \otimes dx\otimes dv  $-a.s. Then by (deterministic) Fubini's theorem, the following relation is easily obtained: for every $\psi \in \mathcal{C}_c^\infty(\bbr^{2d})$, 

\begin{align*}
\int_{\bbr^{2d}} f_t \psi \ dvdx &= \int_{\bbr^{2d}} f^{in} \psi \ dvdx + \int_0^t \int_{\bbr^{2d}} f_s \left( v \cdot \nabla_x \psi + F_a[f_s] \cdot \nabla_v \psi \right)\ dvdx ds \\
&\hspace{0.2cm} +\sigma  \int_{\bbr^{2d}} \Big(\int_0^t  \nabla_v \cdot [(v-v_c) f_s]  \psi \ dW_s \Big) dvdx\\
& \hspace{0.2cm} +\frac{\sigma^2}{2} \int_0^t \int_{\bbr^{2d}} (v-v_c) f_s \cdot \Big[ \nabla_v \big( (v-v_c) \cdot \nabla_v \psi \big) \Big] dvdxds  \quad \mbox{a.s.~ $\omega \in \Omega$}.
\end{align*}
Since $f_t$ is in $L^\infty(\Omega; \mathcal{C}([0,T];\mathcal{C}^2(\bbr^{2d})))$ and compactly supported, we have
\[\int_{\bbr^{2d}}  \left(\int_0^t \Big|\nabla_v \cdot[ (v-v_c)f_s] \psi \Big|^2 ds \right)^{1/2} dv dx <\infty, \quad \mbox{a.s. } \ \omega \in \Omega. \]
Then, we can use the stochastic Fubini theorem (see \cite{V} and references therein) and deterministic Fubini's theorem to get 
\begin{align*} 
\begin{aligned}
& \int_{\bbr^{2d}} \Big(\int_0^t  \nabla_v \cdot [(v-v_c) f_s]  \psi \ dW_s \Big) dvdx \\
& \hspace{0.5cm} =  \int_0^t  \Big( \int_{\bbr^{2d}}\nabla_v \cdot [(v-v_c) f_s]  \psi \ dvdx \Big) dW_s = -  \int_0^t   \Big(  \int_{\bbr^{2d}}[(v-v_c) f_s] \cdot \nabla_v  \psi \ dvdx \Big) dW_s.
\end{aligned}
\end{align*}
This implies our desired result.
\end{proof}
\begin{remark} 
1. If a strong solution $f_t$ to \eqref{A-2} satisfies conditions in Lemma \ref{L2.2}, then $f_t$ satisfies the relation \eqref{B-0-3}. \newline
2. If $f_t$ is a classical solution to \eqref{A-2}, we may use Lemma \ref{L2.2} in \cite{Chow} to obtain that the It\^o relation \eqref{B-0-3} is equivalent to \eqref{B-0}.\\
\end{remark}

\subsection{Quantitative estimates for classical solutions} We provide several properties of classical solutions $f$ to \eqref{A-2}. First, we study the propagation of velocity moments along the stochastic flow of \eqref{B-0}$_1$. For a random density function $f_t$, we set velocity moments:
\begin{equation} \label{B-1}
M_0(t) := \int_{\bbr^{2d}} f_t dvdx, \quad M_1(t) := \int_{\bbr^{2d}} v f_t dvdx, \quad 
M_2(t) :=  \int_{\bbr^{2d}} |v|^2 f_t dvdx, \quad t \geq 0.
\end{equation}
Consider the following stochastic characteristics $\varphi_t(x,v) := (X_t(x,v), V_t(x,v))$:
\begin{equation}\label{char}
\begin{cases}
\displaystyle dX_t = V_t dt, \quad t > 0,\\
\displaystyle dV_t = \left(F_a[f_t](X_t, V_t) \right)dt + \sigma (v_c-V_t) \circ dW_t,
\end{cases}
\end{equation}
subject to the initial data:
\[(X_0(x,v), V_0(x,v)) = (x,v). \]
Note that if $f_t$ is compactly supported in $x$ and $v$, and satisfies the regularity condition for classical solutions, the system \eqref{char} has a unique solution and the family $\{\varphi_{s,t}(x,v) := \varphi_t (\varphi^{-1}_s(x,v))\}$, $0\le s \le t \le T$, forms a stochastic flow of smooth diffeomorphisms (we refer to Lemma 4.1 in Chapter 2 of \cite{Chow} for details).   Furthermore, we define the functionals that measure spatial and velocity supports of $f_t$, respectively:
\begin{align*}
&\mathcal{X}(t) := \sup\{ |x| \ : \ f_t(x,v) \ne 0 \quad \mbox{for some}~~v \in \bbr^d\},\\
&\mathcal{V}(t) := \sup\{ |v| \ : \ f_t(x,v) \ne 0 \quad \mbox{for some}~~x \in \bbr^d\}.
\end{align*}

\begin{lemma}\label{L2.3} 
Let $f_t$ be a classical solution to \eqref{B-0} which is compactly supported in $x$ and $v$ and satisfies 
\[ M_0(0) = 1, \quad M_1(0) = 0. \]
Then, we have
\[ M_0(t) = 1, \quad M_1(t) = 0, \quad M_2(t) \le M_{2}(0) \exp\left(-2\int_0^t {\bar\phi}(2\mathcal{X}(s))ds -2\sigma W_t\right), \quad t \geq 0.\]
\end{lemma}
\begin{proof}  
\noindent $\bullet$~(Conservation of mass):  It follows from Remark 3.1 that
\begin{align}
\begin{aligned} \label{B-1-0}
 f_t(x,v) &= f^{in}(x,v) \hspace{-0.05cm} - \hspace{-0.1cm}\int_0^t \hspace{-0.1cm} \Big(v \cdot \nabla_x f_s  + \nabla_v \cdot (F_a[f_s]f_s) \Big) ds  + \sigma \hspace{-0.1cm}\int_0^t \hspace{-0.1cm} \Big(\nabla_v \cdot( (v-v_c) f_s) \Big) dW_s\\
& \quad +  \frac{\sigma^2}{2} \int_0^t  \nabla_v \cdot \Big[(v-v_c) \nabla_v \cdot \big((v-v_c) f_s\big) \Big] ds .
 \end{aligned}
\end{align}
We integrate \eqref{B-1-0} over $(x,v) \in \bbr^{2d}$ to get
\begin{align*}
\begin{aligned}
&\int_{\bbr^{2d}} f_t(x,v) dv dx \\
& \hspace{0.5cm} =  \int_{\bbr^{2d}}  f^{in}(x,v) dv dx -\int_{\bbr^{2d}}  \int_0^t \Big(v \cdot \nabla_x f_s  + \nabla_v \cdot (F_a[f_s]f_s) \Big) ds  dv dx \\
& \hspace{0.7cm} + \sigma  \int_{\bbr^{2d}}\bigg[ \int_0^t \Big(\nabla_v \cdot( (v-v_c) f_t) \Big) dW_s\bigg] dv dx \\
&\hspace{0.7cm} +  \frac{\sigma^2}{2}\int_{\bbr^{2d}} \int_0^t  \nabla_v \cdot \Big[(v-v_c) \nabla_v \cdot \big((v-v_c) f_s\big) \Big] ds dvdx\\
& \hspace{0.5cm} =:  \int_{\bbr^{2d}}  f^{in}(x,v) dv dx + {\mathcal I}_{11} + {\mathcal I}_{12}+\mathcal{I}_{13} .
\end{aligned}
\end{align*}
Next, we show that the terms ${\mathcal I}_{1i}$ are zero using deterministic and stochastic Fubini's theorems.  \newline

\noindent $\diamond$ (Estimate of ${\mathcal I}_{11}$ and $\mathcal{I}_{13}$): Since $f_t$ has a compact support in $(x,v)$, we can use deterministic Fubini's theorem to see 
\begin{align*}
{\mathcal I}_{11} +{\mathcal I}_{13} &=  - \int_0^t  \int_{\bbr^{2d}} \Big( \nabla_x \cdot(v f_s)  + \nabla_v \cdot (F_a[f_s]f_s) \Big)   dv dx ds\\
& \quad +\frac{\sigma^2}{2}\int_0^t \int_{\bbr^{2d}}   \nabla_v \cdot \Big[(v-v_c) \nabla_v \cdot \big((v-v_c) f_s\big) \Big] dvdx dx\\
&= 0.
 \end{align*}

\noindent $\diamond$ (Estimate of ${\mathcal I}_{12}$):  As in the proof of Lemma \ref{L2.2}, we can use the stochastic Fubini theorem to get 
\[  {\mathcal I}_{12} = \int_0^t  \Big( \int_{\bbr^{2d}}  \nabla_v \cdot( (v-v_c) f_t) dvdx \Big) dW_s = 0. \] 

\vspace{0.2cm}
\noindent $\bullet$~(Conservation of momentum): In this case, we multiply $v$ to \eqref{B-1-0} and use the same argument for conservation of mass to derive 
\[ M_1(t) = M_1(0) = 0, \quad t \geq 0. \]

\noindent $\bullet$~(Dissipation estimate): We multiply \eqref{B-1-0} by $|v|^2$ and use stochastic Fubini's theorem to have

\begin{equation}\label{B-1-1} 
dM_2(t)  =  \Big( 2\sigma^2 M_2(t) + \int_{\bbr^{2d}} 2v \cdot F_a[f_s]f_s dvdx \Big) dt - 2\sigma  M_2(t) dW_t,
\end{equation}

where we used the relation $M_1(t) =0$.\\

\noindent We use \eqref{B-1-1} to get 
\begin{align*}
\begin{aligned}
M_2(t) &= M_{2}(0) + \int_0^t \bigg[ \Big(\int_{\bbr^{2d}} 2v \cdot F_a[f_s]f_s dvdx\Big) +2\sigma^2 M_2(s) \bigg] ds-2\sigma \int_0^t M_2(s) dW_s\\
& = M_{2}(0) + 2\int_0^t \int_{\bbr^{4d}} \phi(x_*-x)(v_*-v)\cdot v f_s(x_*,v_*)f_s(x,v)dv_* dx_*dvdxds \\
& \hspace{0.5cm} + 2\sigma^2 \int_0^t M_2(s) ds - 2\sigma \int_0^t M_2(s) dW_s\\
& = M_{2}(0) - \int_0^t \int_{\bbr^{4d}} \phi(x_*-x)|v-v_*|^2 f_s(x_*,v_*)f_s(x,v)dv_* dx_*dvdxds \\
& \hspace{0.5cm} + 2\sigma^2 \int_0^t M_2(s) ds -2 \sigma \int_0^t M_2(s)  dW_s\\
& \leq M_{2}(0) -  \int_0^t {\bar\phi}(2\mathcal{X}(s)) \left[\int_{\bbr^{4d}}|v-v_*|^2f_s(x_*,v_*)f_s(x,v)dv_* dx_*dvdx\right]ds \\
& \hspace{0.5cm} + 2\sigma^2 \int_0^t M_2(s) ds -2\sigma \int_0^t M_2(s) dW_s\\
&\leq M_{2}(0)  -2 \int_0^t \left({\bar\phi}(2\mathcal{X}(s))-\sigma^2\right)M_2(s)ds -2\sigma \int_0^t M_2(s) dW_s.
\end{aligned}
\end{align*}
Then we use Lemma \ref{LA-1} and Lemma \ref{LA-2} to get
\[M_2(t) \le M_{2}(0)\exp\left(-2\int_0^t {\bar\phi}(2\mathcal{X}(s))ds -2\sigma W_t\right). \]
\end{proof}
\begin{remark} 
In Lemma 3.3, we observe that the first momentum is preserved. Thus, without loss of generality, we may assume that $v_c(t) = 0$.
\end{remark}

\vspace{0.2cm}

Next, we discuss the size of spatial and velocity supports of $f_t$. 
\begin{lemma}\label{L4.1}
The support functionals $\mathcal{X}$ and $\mathcal{V}$ satisfy the following estimates:
\begin{align*}
&\mathcal{X}(t) \leq \mathcal{X}_0 + \sqrt{2}\int_0^t\left(\mathcal{V}_0 +\phi_M \sqrt{dM_2(0)}s\right)\exp\left[ -\int_0^s \bar\phi(2\mathcal{X}(\tau))d\tau - \sigma W_s\right]ds, \quad t \geq 0, \\
&\mathcal{V}(t) \leq  \sqrt{2}\left(\mathcal{V}_0 +\phi_M \sqrt{dM_2(0)}t\right)\exp\left[ -\int_0^t \bar\phi(2\mathcal{X}(s))ds - \sigma W_t\right].
\end{align*}
Moreover, if $\phi_m>0$, then
\begin{align*}
&\mathcal{X}(t) \leq \mathcal{X}_0 + \sqrt{2}\int_0^t\left(\mathcal{V}_0 +\phi_M \sqrt{dM_2(0)}s\right)\exp( -\phi_m s - \sigma W_s)ds, \quad t \geq 0, \\
&\mathcal{V}(t) \leq  \sqrt{2}\left(\mathcal{V}_0 +\phi_M \sqrt{dM_2(0)}t\right)\exp( -\phi_m t - \sigma W_t) .
\end{align*}

\end{lemma}
\begin{proof} 
\noindent First, we consider the case when $\phi_m > 0$ may not hold.  \newline

\noindent $\diamond$~(Estimate of ${\mathcal V}$):  Note that the stochastic characteristics $(X_t, V_t) = \{(x^i_t, v^i_t)\}_{i=1}^d$ starting from $(x,v) \in \mbox{supp}f^{in}$ satisfy
\begin{equation} \label{L3-2.1}
\begin{cases}
\displaystyle dx^i_t = v^i_t dt, \quad 1 \le i \le d,\\
\displaystyle dv^i_t = \left(\Big(F_a[f_t](X_t, V_t)\Big)^i + \frac{1}{2}\sigma^2 v^i_t\right)dt - \sigma v^i_t dW_t.
\end{cases}
\end{equation}
\noindent Now, we rewrite \eqref{L3-2.1}$_2$ to have

\begin{align}
\begin{aligned} \label{D-0-0}
dv_t^i &= \Bigg[\left(-\int_{\bbr^{2d}}\phi(x_* - X_t)f_t(x_*,v_*)dv_*dx_* + \frac{1}{2} \sigma^2 \right)v_t^i \\
&\hspace{1.3cm} + \int_{\bbr^{2d}}\phi(x_* - X_t)v_*^i f_t(x_*,v_*)dv_*dx_* \Bigg]dt - \sigma v_t^i dW_t.
\end{aligned}
\end{align}
Thus, we apply Lemma \ref{LA-1} and Lemma \ref{L2.3} to \eqref{D-0-0} to get
\begin{align*}
|v_t^i| &= \Bigg|v_0^i\exp\left[ -\int_0^t \left\{\int_{\bbr^{2d}}\phi(x_* - X_s)f_s(x_*, v_*)dv_*dx_* \right\}ds - \sigma W_t\right]\\
&\quad + \int_0^t \left\{\int_{\bbr^{2d}}\phi(x_* - X_s)v_*^i f_s(x_*, v_*)dv_*dx_*\right\} \\
&\hspace{1cm}\times\exp\left[ -\int_s^t \left\{\int_{\bbr^{2d}}\phi(x_* - X_\tau)f_\tau(x_*, v_*)dv_*dx_* \right\}d\tau - \sigma (W_t-W_s)\right]ds\Bigg|\\
&\le |v_0^i|\exp\left[ -\int_0^t {\bar\phi}(2\mathcal{X}(s))ds - \sigma W_t\right]\\
&\quad + \phi_M \int_0^t \sqrt{M_2(s)}\exp\left[ -\int_s^t {\bar\phi}(2\mathcal{X}(\tau))d\tau - \sigma( W_t-W_s)\right] ds\\
&\le \left(|v_0^i| + \phi_M\sqrt{M_2(0)}t \right)\exp\left[ -\int_0^t {\bar\phi}(2\mathcal{X}(s))ds - \sigma W_t\right].
\end{align*}
Hence, we have
\begin{align*}
|V_t|^2 &= \sum_{i=1}^d |v_t^i|^2 \le \sum_{i=1}^d\left(|v_0^i| + \phi_M\sqrt{M_2(0)}t \right)^2\exp\left[ -2\int_0^t {\bar\phi}(2\mathcal{X}(s))ds - 2\sigma W_t\right]\\
&\le 2 \left(|V_0|^2 + d\phi_M^2 M_2(0) t^2 \right)^2\exp\left[ -2\int_0^t {\bar\phi}(2\mathcal{X}(s))ds - 2\sigma W_t\right],
\end{align*}
where we used Young's inequality, and this yields
\[ \mathcal{V}(t) \leq  \sqrt{2}\left(\mathcal{V}_0 +\phi_M \sqrt{dM_2(0)}t\right)\exp\left[ -\int_0^t {\bar\phi}(2\mathcal{X}(s))ds - \sigma W_t\right] .  \]
This gives the desired estimate. \newline

\noindent $\diamond$~(Estimate of ${\mathcal X}$):~We use It\^o's formula and the Cauchy-Schwarz inequality to get
\[ d|X_t|^2 =2X_t \cdot dX_t + dX_t \cdot dX_t = 2X_t \cdot V_t  dt \leq 2 |X_t| \cdot |V_t| dt.  \]
This and the estimates for $\mathcal{V}(t)$ yield
\[ \frac{d|X_t|}{dt} \leq |V_t| \leq\ \sqrt{2}\left(\mathcal{V}_0 +\phi_M \sqrt{dM_2(0)}t\right)\exp\left[ -\int_0^t {\bar\phi}(2\mathcal{X}(s))ds - \sigma W_t\right]. \] 
We integrate the above differential inequality to get 
\[ |X_t| \leq |X_0| + \sqrt{2}\int_0^t\left(\mathcal{V}_0 +\phi_M \sqrt{dM_2(0)}s\right)\exp\left[ -\int_0^s {\bar\phi}(2\mathcal{X}(\tau))d\tau - \sigma W_s\right]ds,   \] 
and this implies our desired estimate for $\mathcal{X}(t)$.
\vspace{0.4cm}

\noindent When $\phi_m >0$, we can use $\phi_m \le \phi(2\mathcal{X}(s))$ to get the desired results. \newline
\end{proof}

\begin{remark} 
Here, we discuss the necessity of the lower bound condition $\phi_m >0$ for flocking estimates. As observed in \cite{H-T}, the equation \eqref{B-0} without noise exhibits flocking without the condition $\phi_m >0$, but it was attainable since the sizes of $x$- and $v$-supports increase at most in an algebraic order, which is not the case for \eqref{B-0} due to the exponential Wiener process. Here, it is well known that
\[
\limsup_{t \to \infty} \frac{W_t}{\sqrt{2t \log \log t}} = 1, \quad \mbox{for a.s.} \ \omega \in \Omega.
\]

\noindent Thus, for the pathwise flocking estimate, we require
\begin{equation}\label{rmk-1}
\limsup_{t \to \infty}\frac{\int_0^t {\bar\phi}(2\mathcal{X}(s))ds}{\sqrt{t \log \log t}}  = 0, \quad \mbox{for a.s. }\omega \in \Omega.
\end{equation}

However, as observed in Lemma \ref{L4.1}, it becomes difficult to estimate $\mathcal{X}(t)$ without the lower bound assumption $\phi_m >0$. Accordingly, it is hard to find a condition weaker than $\phi_m >0$ which entails the estimate \eqref{rmk-1}.
\end{remark}

\vspace{0.2cm}

\noindent Now, we are ready to state the stability results for \eqref{A-2}.  
\begin{theorem}\label{T3.1}
\emph{($L^\infty$-stability)} Let $f_t$ and $\tilde{f}_t$ be two classical solutions to \eqref{A-2} corresponding to regular initial data $f^{in}$ and $\tilde{f}^{in}$, respectively, which are compactly supported in $x$ and $v$. Moreover, let $\varphi_t = \varphi_t(x,v)$ and $\tilde{\varphi}_t = \tilde{\varphi}_t(x,v)$ be the stochastic characteristics associated to $f$ and $\tilde{f}$, respectively. Then, we have
\[ \|f_t-\tilde{f}_t\|_{\mathcal{C}^0}^2 + \|\varphi_t - \tilde{\varphi}_t\|_{\mathcal{C}^0}^2 \le \mathcal{D}_t \|f^{in} - \tilde{f}^{in}\|_{\mathcal{C}^0}^2,\]
where $\mathcal{D}_t$ is a non-negatvie process with continuous sample paths and
\[ 
\|\varphi_t - \tilde{\varphi}_t\|_{\mathcal{C}^0} := \sup\left\{|\varphi_t(x,v) - \tilde{\varphi}_t(x,v)| \ : \ (x,v) \in (\mbox{supp} f^{in})\cup (\mbox{supp} \tilde{f}^{in})\right\}.
\]
\end{theorem}
\begin{proof}
Since the proof is rather lengthy, we postpone it to Appendix B.
\end{proof}

\section{Global existence and asymptotic dynamics of strong solutions}\label{sec:4}
\setcounter{equation}{0}

In this section, we provide global existence and asymptotic flocking estimates for strong solutions to \eqref{A-2}. Here, we show our desired estimates for \eqref{B-0} corresponding to regularized initial data. Then, based on the stability estimates for classical solutions that we obtained in the previous section, we conclude that solutions to \eqref{B-0} with regularized initial data converge to a strong solution to \eqref{B-0}. Moreover, we show that a strong solution obtained as above satisfies the asymptotic flocking estimates.

Let $f^{in, \varepsilon}$ be a smooth mollification of the given initial datum $f^{in}$ satisfying the framework $({\mathcal F})$. Then, consider the Cauchy problem \eqref{B-0} with these regularized initial data:
\begin{align}
\begin{aligned} \label{C-0}
& \partial_t f^{\varepsilon}_t + v \cdot \nabla_x f^{\varepsilon}_t + \nabla_v \cdot (F_a[f^{\varepsilon}_t]f^{\varepsilon}_t)  = \sigma \nabla_v \cdot (v f^{\varepsilon}_t) \circ \dot{W}_t, \quad  (t, x,v)\in \bbr_+ \times \bbr^{2d}, \\
& f^{\varepsilon}_0(x,v) = f^{in,\varepsilon}(x,v).
\end{aligned}
\end{align}
Note that due to the framework $(\mathcal{F})$, the initial datum $f^{in}$ and its partial derivatives up to order $k$ are uniformly continuous on $\bbr^{2d}$ and there exists a constant $R_0>0$, such that
\[ \quad \mbox{supp} f^{in} \subseteq B_{R_0}(0),\]
where $B_{R_0}(0)$ is a ball of radius $R_0$ centered at $0 \in \bbr^{2d}$. As mentioned above, we use a mollifier to obtain a family of regularized initial data $f^{in,\varepsilon} \in \mathcal{C}^\infty(\bbr^{2d})$, $\varepsilon \in (0,1)$, so that the regularized datum satisfies the following conditions: \newline
\begin{itemize}
\item
($\mathcal{F}^{\varepsilon}1$):~$\{f^{in,\varepsilon} \}$ are nonnegative, compactly supported, uniformly converge to $f^{in}$ in $\mathcal{C}^0(\bbr^{2d})$ and 
\[\|f^{in, \varepsilon}\|_{W^{k,\infty}} \le \|f^{in}\|_{W^{k,\infty}}.\\ \]

\vspace{0.2cm}

\item
($\mathcal{F}^{\varepsilon}2$):~$\{M_{2}^\varepsilon\}(0)$ is uniformly bounded with respect to $\varepsilon$ and converges to $M_{2}(0)$ as $\varepsilon \to 0$.

\vspace{0.2cm}

\item
($\mathcal{F}^{\varepsilon}3$):~The zeroth and first moment of $f^{in, \varepsilon}$ are initially constrained:
\[\int_{\bbr^{2d}} f^{in,\varepsilon}dx dv = 1, \quad \int_{\bbr^{2d}} vf^{in,\varepsilon}dvdx = 0. \]

\vspace{0.2cm}

\item
($\mathcal{F}^\varepsilon 4$):~$f^{in,\varepsilon}$  has a compact support in $x$ and $v$, and satisfy
\[\mbox{supp}f^{in,\varepsilon} \subseteq B_{R_0+1}(0).\\ \]
\end{itemize}
In the following three subsections, we will provide a global existence for system \eqref{C-0}. 

\subsection{Construction of approximate solutions} \label{sec:3.1}
In this subsection, we provide a sequence of approximate solutions to \eqref{C-0} using successive approximations. \newline

\noindent First,  the zeroth iterate $f_t^{0, \varepsilon}$ is simply defined as the mollified initial datum:
\[ f_t^{0, \varepsilon}(x,v) := f^{in, \varepsilon}(x,v), \quad (x, v) \in \bbr^{2d}. \]
For $n \geq 1$, suppose that the $(n-1)$-th iterate $f_t^{n-1, \varepsilon}$ is given. Then, the $n$-th iterate is defined as the solution to the linear equation with fixed initial datum: 
\begin{equation} \label{C-1}
\begin{cases}
\displaystyle \partial_t f_t^{n,\varepsilon} + v \cdot \nabla_x f_t^{n,\varepsilon} + \nabla_v \cdot (F_a[f_t^{n-1,\varepsilon}]f_t^{n,\varepsilon}) = \sigma \nabla_v \cdot (v f_t^{n,\varepsilon}) \circ \dot{W}_t, \ \ n \geq 1, \\
\displaystyle f_0^{n,\varepsilon}(x,v) = f^{in,\varepsilon}(x,v).
\end{cases}
\end{equation}  
The linear system \eqref{C-1} can be solved by the method of stochastic characteristics. 
Let $\varphi_t^{n,\varepsilon}(x,v) := (X_t^{n,\varepsilon}(x,v), V_t^{n,\varepsilon}(x,v))$ be the forward stochastic characteristics, which is a solution to the following SDE:
\begin{equation} \label{C-2}
\begin{cases}
\displaystyle dX_t^{n,\varepsilon} = V_t^{n,\varepsilon}dt,\\
\displaystyle dV_t^{n,\varepsilon} = F_a[f_t^{n-1,\varepsilon}](X_t^{n,\varepsilon}, V_t^{n,\varepsilon})dt - \sigma V_t^{n,\varepsilon} \circ dW_t, \\
(X_t^{n,\varepsilon}(0), V_t^{n,\varepsilon}(0)) = (x,v) \in \mbox{supp}f^{in,\varepsilon}.
\end{cases}
\end{equation}
Note that the SDE \eqref{C-2} is equivalent to the following It\^o SDE \cite{Ev}:

\begin{equation} \label{C-2-1}
\begin{cases}
\displaystyle dX_t^{n,\varepsilon} = V_t^{n,\varepsilon}dt,\\
\displaystyle dV_t^{n,\varepsilon} = \Big(F_a[f_t^{n-1,\varepsilon}](X_t^{n,\varepsilon}, V_t^{n,\varepsilon}) + \frac{\sigma^2}{2} V_t^{n,\varepsilon} \Big)dt - \sigma V_t^{n,\varepsilon} dW_t, \\
(X_t^{n,\varepsilon}(0), V_t^{n,\varepsilon}(0)) = (x,v) \in \mbox{supp}f^{in,\varepsilon}.
\end{cases}
\end{equation}

Here, we can deduce from  Lemma 3.1 and Theorem 3.2 in \cite{Chow} and our framework that for any $m \ge 3$, \eqref{C-2} has a unique solution $f_t^{n,\varepsilon}$ which is a $\mathcal{C}^m$-semimartingale for every $n \ge 0$ and the characteristics \eqref{C-2} becomes a $\mathcal{C}^m$-diffeomorphism. Then, $f_t^{n,\varepsilon}$ can also be represented by the following integral formula:
\begin{equation}\label{C-2-2}
f_t^{n,\varepsilon}(\varphi_t^{n,\varepsilon}(x,v)) = f^{in,\varepsilon}(x,v)\exp\Big[ - \int_0^t \nabla_v \cdot F_a[f^{n-1,\varepsilon}](s,\varphi_s^{n,\varepsilon}(x,v))ds + d\sigma W_t \Big].
\end{equation}
Note that if $f^{in,\varepsilon}$ is nonnegative, then surely $f_t^{n,\varepsilon}$ is also nonnegative as well. Before we finish this subsection, we also remark that the linear, first-order Stratonovich equation \eqref{C-1} is equivalent to the following parabolic It\^o equation (see Corollary 3.3. in \cite{Chow}):

\begin{equation} \label{C-1-1}
\begin{cases}
\displaystyle \partial_t f_t^{n,\varepsilon} + v \cdot \nabla_x f_t^{n,\varepsilon} + \nabla_v \cdot (F_a[f_t^{n-1,\varepsilon}]f_t^{n,\varepsilon})\\
\displaystyle \hspace{1cm} = \sigma \nabla_v \cdot (v f_t^{n,\varepsilon}) \dot{W}_t +\frac{\sigma^2}{2} \nabla_v \cdot \Big[ v \nabla_v \cdot (v f_t)\Big], \ \ n \ge 1,\\
\displaystyle f_0^{n,\varepsilon}(x,v) = f^{in,\varepsilon}(x,v).
\end{cases}
\end{equation}

\subsection{Estimates on approximate solutions}
In this subsection, we provide several estimates for the approximate solutions for \eqref{C-1}. To be more precise, we would try to obtain $n$ and $\varepsilon$-independent estimates for the later sections. Before  we move on, we define $p$-th  velocity moments $M^{n,\varepsilon}_{p}(t)$, $p=0,1,2$:
\begin{align*}
&M_{0}^{n,\varepsilon}(t) := \int_{\bbr^{2d}} f_t^{n,\varepsilon}dvdx, \quad M_1^{n,\varepsilon}(t) := \int_{\bbr^{2d}} vf_t^{n,\varepsilon}dvdx,\\
&M_2^{n,\varepsilon}(t) := \int_{\bbr^{2d}}|v|^2f_t^{n,\varepsilon}dvdx, \quad M_p^{n,\varepsilon}(0) := M_{p0}^\varepsilon.
\end{align*}
Before we provide the uniform estimates for the $p$-th ($p=0,1,2$) moments, we set 
\begin{equation}\label{C-3}
M_{20}^\infty := \sup_{\varepsilon \in (0,1)} M_{2}^\varepsilon(0),  \qquad \gamma:= \max\{M_{20}^\infty, \phi_M\}.
\end{equation}
We also present a technical lemma from \cite{B-D-G-M} for a later discussion.
\begin{lemma}\cite{B-D-G-M}\label{L4.1}
Let $T \in (0,\infty]$ and $(a_n)_{n \in \bbn}$ be a sequence of nonnegative continuous functions on $[0,T]$ satisfying
\[a_n(t) \le A + B\int_0^t a_{n-1}(s)ds + C\int_0^t a_n(s)ds, \quad t \in [0,T], \quad n \ge 1, \]
where $A$, $B$ and $C$ are nonnegative constants.\\
\begin{enumerate}
\item
If $A=0$, there exists a constant $\Lambda\ge 0$ depending on $B$, $C$ and $\displaystyle \sup_{t \in [0,T]}a_0(t)$ such that
\[a_n(t) \le \frac{(\Lambda t)^n}{n!}, \quad t \in [0,T], \quad n \in \bbn. \]
\item
If $A>0$ and $C=0$, there exists a constant $\Lambda \ge  0$ depending on $A$, $B$ and  $\displaystyle \sup_{t \in [0,T]}a_0(t)$ such that
\[a_n(t) \le \Lambda \exp(\Lambda t), \quad t \in [0,T], \quad n \in \bbn.  \]
\end{enumerate}
\end{lemma}

\begin{remark}
1. In (2) of Lemma \ref{L4.1}, $\Lambda$ can be explicitly written as
\[ \Lambda := \max\left\{ A, \ B, \ \sup_{t\in[0,T]} a_0(t) \right\}.\]
2. We can also use the similar argument to obtain the following estimate for (2):
\[ a_n(t) \le (\Lambda + \kappa_t) \exp(\Lambda t), \quad t \in [0,T], \quad n \in \bbn,\]
where $\Lambda := \max\{A,B\}$ and $\kappa_t := \sup_{0\le s \le t}a_0(s)$.
\end{remark}

\bigskip

\begin{proposition}\label{P4.1}
For every $n \in  \bbn$ and $T \in (0, \infty)$, let $f_t^{n,\varepsilon}$ be a solution to \eqref{C-1}. Then, for any $t \in (0,T)$ we have
\[ M_0^{n,\varepsilon}(t) = 1, \quad M_1^{n,\varepsilon}(t) = 0, \quad M_2^{n,\varepsilon}(t) \le (\gamma + K_t) \exp\{(\gamma+\phi_M)t -2\sigma W_t\},  \]
where $\gamma$ is a constant in \eqref{C-3} and $K_t$ is defined as
\[ K_t := M_{20}^\infty \sup_{0 \le s \le t} \exp(-\phi_M s + 2\sigma W_s ).\]
\end{proposition}
\begin{proof} Note that $f_t^{n,\varepsilon}$ satisfies relation \eqref{C-1-1} and $f_t^{n,\varepsilon}$ is compactly supported in $x$ and $v$, since $f^{in,\varepsilon}$ is compactly supported in the phase space and $\varphi_t^{n,\varepsilon}$ is a $\mathcal{C}^m$-diffeomorphisim. Thus, we may follow the arguments in Lemma \ref{L2.3} to derive the conservation estimates.   \newline

\noindent For the dissipation estimate of $M_2^{n,\varepsilon}$, we use a similar argument to Lemma \ref{L2.3} to have 
\begin{align*}
M_2^{n,\varepsilon}(t) & = M_{2}^\varepsilon(0) + 2\int_0^t \int_{\bbr^{4d}} \phi(x_*-x)(v_*-v)\cdot v f_s^{n-1,\varepsilon}(x_*,v_*)f_s^{n,\varepsilon}(x,v)dv_* dx_*dvdxds \\
& \quad +2\sigma^2 \int_0^t  M_2^{n,\varepsilon}(s) ds - 2\sigma \int_0^t M_2^{n,\varepsilon}(s)  dW_s\\
& \le M_{2}^\varepsilon(0) + 2\int_0^t \int_{\bbr^{4d}} \phi(x_*-x)v_*\cdot v f_s^{n-1,\varepsilon}(x_*,v_*)f_s^{n,\varepsilon}(x,v)dv_* dx_*dvdx\\
&  \quad +2\sigma^2 \int_0^t  M_2^{n,\varepsilon}(s) ds -2 \sigma \int_0^t M_2^{n,\varepsilon}(s)  dW_s\\
& \le  M_{2}^\varepsilon(0) + \phi_M \int_0^t M_2^{n-1,\varepsilon}(s)ds + (\phi_M + 2\sigma^2)\int_0^t M_2^{n,\varepsilon}(s)ds - 2\sigma\int_0^t M_2^{n,\varepsilon}(s)dW_s,
\end{align*}
where we used Young's inequality on the second inequality. In a differential form, we have
\begin{equation}\label{P3-1.1}
dM_2^{n,\varepsilon}(t) \le \left\{ \phi_M M_2^{n-1,\varepsilon}(t) + (\phi_M + 2\sigma^2) M_2^{n,\varepsilon}(t) \right\} dt -2\sigma M_2^{n,\varepsilon}(t)dW_t.
\end{equation}
Then, it follows from \eqref{P3-1.1} and comparison theorem (in Lemma \ref{LA-2}) that  
\[ M_2^{n,\varepsilon}(t) \le X_t, \]
where the process $X_t$ satisfiesf
\[
\begin{cases}
dX_t = \left\{ \phi_M M_2^{n-1,\varepsilon}(t) + (\phi_M + 2\sigma^2) X_t \right\} dt -2\sigma X_tdW_t, \quad t > 0, \\
X_0 = M^{\varepsilon}_{2}(0).
\end{cases}
\]
\noindent It follows from Lemma \ref{LA-1} that $X_t$ can be represented as
\[X_t = X_0 \exp(\phi_M t -2\sigma W_t) + \phi_M \int_0^t \exp\{\phi_M(t-s) - 2\sigma(W_t - W_s)\} M_2^{n-1,\varepsilon}(s)ds. \]
\noindent This implies
\[M_2^{n,\varepsilon}(t) \le M^{\varepsilon}_{2}(0)\exp(\phi_M t -2\sigma W_t) + \phi_M \int_0^t \exp\{\phi_M(t-s) - 2\sigma(W_t - W_s)\} M_2^{n-1,\varepsilon}(s)ds. \]
\noindent Now, we set 
\[ a_n(t) := M_2^{n,\varepsilon}(t) \exp \{-\phi_M t + 2\sigma W_t\}. \]
Then, it satisfies
\[ a_{n+1}(t) \le M_{20}^\varepsilon + \phi_M \int_0^t a_n(s)ds.\]
\noindent We use Lemma \ref{L4.1} in the way from Remark 4.1 to get
\[a_n(t) \le (\gamma +K_t)e^{\gamma t}, \quad t \in (0,T). \]
This yields the desired result.
\end{proof}
\noindent We also provide uniform estimates for the stochastic characteristic flows.
\begin{proposition}\label{P3.2}
For each $ n \in \bbn$ and $T \in (0, \infty)$, let $(X_t^{n,\varepsilon}, V_t^{n,\varepsilon})$ be the stochastic characteristic flow for \eqref{C-1} with the initial data:
\[ (X_0^{n,\varepsilon}, V_0^{n,\varepsilon}) = (x,v) \in \mbox{supp} f^{in,\varepsilon}. \]
Then for $t \in (0,T)$, we have
\begin{align*}
&(i)~~ |V_t^{n,\varepsilon}|^2 \le \left\{|v|^2 + \phi_M \int_0^t (\gamma +K_s)\exp(\gamma s) ds \right\} \exp(\phi_M t - 2\sigma W_t). \\
&(ii)~~ |X_t^{n,\varepsilon}|^2 \le 2 \left( |x|^2 + t \int_0^t \left\{|v|^2 + \phi_M \int_0^s (\gamma +K_\tau)\exp(\gamma\tau)d\tau \right\} \exp(\phi_M s - 2\sigma W_s) ds \right).
\end{align*}
\end{proposition}
\begin{proof}
\noindent (i)~ It follows from It\^o's lemma and \eqref{C-2-1} that
\begin{align*}
d|V_t^{n,\varepsilon}|^2  &= 2V_t^{n,\varepsilon} \cdot dV_t^{n,\varepsilon} + dV_t^{n,\varepsilon}\cdot dV_t^{n,\varepsilon}\\
&= 2\left( F_a[f^{n-1,\varepsilon}_t](X_t^{n,\varepsilon}, V_t^{n,\varepsilon})\cdot V_t^{n,\varepsilon} +\sigma^2 |V_t^{n,\varepsilon}|^2 \right)dt -2\sigma |V_t^{n,\varepsilon}|^2dW_t\\
&\le \left(2 \int_{\bbr^{2d}}\phi(x_* - X_t^{n,\varepsilon})(v_* \cdot V_t^{n,\varepsilon}) f^{n-1,\varepsilon}_t(x_*,v_*)dv_* dx_* + \sigma^2 |V_t^{n,\varepsilon}|^2 \right)dt - 2\sigma d |V_t^{n,\varepsilon}|^2dW_t\\
&\le \left( \phi_M M_2^{n-1,\varepsilon}(t) +  (\phi_M + 2\sigma^2)|V_t^{n,\varepsilon}|^2\right) dt - 2\sigma  |V_t^{n,\varepsilon}|^2dW_t,
\end{align*}
where $ dV_t^{n,\varepsilon}\cdot dV_t^{n,\varepsilon}$ denotes a handy notation for a quadratic variation of $V_t^{n,\varepsilon}$. \newline

\noindent We use Proposition \ref{P4.1} and Lemmas \ref{LA-1}-\ref{LA-2} to get 
\begin{align*}
|V_t^{n,\varepsilon}|^2 &\le |v|^2\exp(\phi_M t - 2\sigma W_t) + \phi_M \int_0^t \exp\{\phi_M(t-s) - 2\sigma(W_t - W_s)\} M_2^{n-1,\varepsilon}(s)ds\\
& \le \left\{|v|^2 + \phi_M  \int_0^t (\gamma +K_s)\exp(\gamma s)ds \right\} \exp(\phi_M t - 2\sigma W_t).
\end{align*}

\bigskip

\noindent (ii)~For the estimate of spatial process, we use Cauchy-Schwarz inequality to get

\begin{align*}
|X_t^{n,\e}|^2 &\le \left( |x|^2 + \int_0^t |V_s^{n,\e}|^2 ds \right)^2 \le 2 \left( |x|^2 + t \int_0^t |V_s^{n,\e}|^2 ds \right)\\
& \le 2 \left( |x|^2 + t \int_0^t \left\{|v|^2 + \phi_M \int_0^s (\gamma +K_\tau)\exp(\gamma\tau)d\tau \right\} \exp(\phi_M s - 2\sigma W_s) ds \right).
\end{align*}
This yields the desired result.

\end{proof}
As a corollary of Proposition \ref{P3.2},  we have estimates for the sizes of velocity and spatial supports:  We set 
\begin{align*}
\begin{aligned}
\mathcal{X}^{n,\varepsilon}(t) &:= \sup\{ |x| \ : \ f_t^{n,\varepsilon}(x,v) \ne 0 \quad \mbox{for some}~v \in \bbr^d\}, \cr
\mathcal{V}^{n,\varepsilon}(t) &:= \sup\{ |v| \ : \ f_t^{n,\varepsilon}(x,v) \ne 0 \quad \mbox{for some}~x \in \bbr^d\}.
\end{aligned}
\end{align*}
\begin{corollary} \label{C3.1}
For each $ n \in \bbn$ and $T \in (0, \infty]$, let $(X_t^{n,\varepsilon}, V_t^{n,\varepsilon})$ be the stochastic characteristic flow for \eqref{C-1} with the initial data:
\[ (X_0^{n,\varepsilon}, V_0^{n,\varepsilon}) = (x,v) \in \mbox{supp} f^{in,\varepsilon}. \]
Then for $t \in (0,T)$, we have
\[  |\mathcal{V}^{n,\varepsilon}(t)| \le |\mathcal{V}^\infty(t)| \quad \mbox{and} \quad |\mathcal{X}^{n,\varepsilon}(t)| \le |\mathcal{X}^\infty(t)|, \]
where $\mathcal{X}^\infty(t)$ and $\mathcal{V}^\infty(t)$ are given by the following relations:
\begin{align*}
& |\mathcal{X}^\infty(t)|^2 \\
&:=  2\left( (R_0+1)^2 + t\int_0^t \left\{(R_0 +1)^2 + \phi_M \int_0^s (\gamma +K_\tau) \exp(\gamma\tau)d\tau \right\} \exp(\phi_M s - 2\sigma W_s) ds\right),\\
& |\mathcal{V}^\infty(t)|^2 :=  \left\{(R_0+1)^2 + \phi_M \int_0^t (\gamma +K_s)\exp(\gamma s) ds\right\} \exp(\phi_M t - 2\sigma W_t).
\end{align*}

\end{corollary}
\begin{proof} It follows from Proposition \ref{P3.2} that
\begin{align*}
\begin{aligned}
|\mathcal{V}^{n,\varepsilon}(t)|^2  &\le \left\{|\mathcal{V}^{n,\varepsilon}(0)|^2 + \phi_M \int_0^t (\gamma +K_s)\exp(\gamma s) ds \right\} \exp(\phi_M t - 2\sigma W_t) \cr
&\le \left\{(R_0+1)^2 + \phi_M \int_0^t (\gamma +K_s)\exp(\gamma s) ds \right\} \exp(\phi_M t - 2\sigma W_t) = |\mathcal{V}^\infty(t)|^2.
\end{aligned}
\end{align*}
This yields the first estimate for velocity support. On the other hand, we also use Proposition \ref{P3.2} to get

\begin{align*}
\begin{aligned}
&|\mathcal{X}^{n,\varepsilon}(t)|^2 \cr
& \le 2\left( |\mathcal{X}^{n,\varepsilon}(0)|^2 + t\int_0^t \left\{|\mathcal{V}^{n,\e}(0)|^2 + \phi_M  \int_0^s (\gamma +K_\tau) \exp(\gamma\tau)d\tau \right\} \exp(\phi_M s - 2\sigma W_s) ds\right)\cr
&\le 2\left( (R_0+1)^2 + t\int_0^t \left\{(R_0 +1)^2 + \phi_M \int_0^s (\gamma +K_\tau) \exp(\gamma\tau)d\tau \right\} \exp(\phi_M s - 2\sigma W_s) ds\right) \cr
&=: |\mathcal{X}^\infty(t)|^2.
\end{aligned}
\end{align*}
\end{proof}
\begin{remark}
Note that $f_t^{n,\varepsilon}$ has compact supports in $x$ and $v$ for every sample path which are bounded uniformly in $n$ and $\varepsilon$.
\end{remark}
\noindent Now, we are ready to state the results on the uniform bound for the sequence $\{f_t^{n,\varepsilon}\}$.
\begin{proposition}\label{P3.3}
For every $n$, $m \in \bbn$ and $ t \in (0,T)$, there exists a nonnegative process $\mathcal{A}^m_t$ which has continuous sample paths and is independent of $n$ and $\varepsilon$ such that
\[\|f_t^{n,\varepsilon}\|_{W^{m,\infty}} \le \mathcal{A}^m_t \cdot \|f^{in,\varepsilon} \|_{W^{m,\infty}}. \] 
\end{proposition}
\begin{proof}
Since the proof is quite lengthy, we postpone it to Appendix C.
\end{proof}
\begin{remark}
It is easy to see that for fixed $t$ and $\omega$, $\mathcal{A}_t^m$ is monotonically increasing with respect to $m$.
\end{remark}
\noindent Next, we prove that the sample paths of approximate solutions become a Cauchy sequence in a suitable functional space.
\begin{proposition}\label{P3.4}
For every $n$ and $t \in (0,T)$, there exists a nonnegative process $\tilde{\mathcal{D}}_t$ which has continuous sample paths and is independent of $n$ and $\varepsilon$ such that
\begin{align}
\begin{aligned} \label{C-5}
& \|f^{n,\varepsilon}_t - f^{n-1,\e}_t\|_{\mathcal{C}^0}^2 + \|\varphi_t^{n,\varepsilon} - \varphi_t^{n-1,\varepsilon}\|_{\mathcal{C}^0}^2 \\
&  \hspace{1cm} \le \tilde{\mathcal{D}}_t  \bigg[ \int_0^t  \Big( \|\varphi_s^{n,\varepsilon} - \varphi_s^{n-1, \varepsilon}\|_{\mathcal{C}^0}^2+ \|f^{n-1,\varepsilon}_s - f^{n-2,\varepsilon}_s\|_{\mathcal{C}^0}^2 \Big) ds \bigg], \quad n \geq 2.
\end{aligned}
\end{align}
\end{proposition} 
\begin{proof} 
Since the proof is almost the same as that of Theorem \ref{T3.1}, we only point out some differences. In the proof of Theorem \ref{T3.1}, we just replace $\mathcal{R}(t)$, $\mathcal{P}(t)$,  $\max(\|f^{in}\|_{L^\infty}, \|\tilde{f}^{in}\|_{L^\infty})$ and $\max(\|f_t\|_{W^{1,\infty}}, \|\tilde{f}_t\|_{W^{1,\infty}})$ by $\mathcal{X}^\infty(t)$, $\mathcal{V}^\infty(t)$, $\|f^{in}\|$ and $\|f^{in}\|_{W^{1,\infty}} \mathcal{A}_t^1$, respectively. Then it becomes our desired estimate and hence, we can actually get

\begin{align*}
&\|f_t^{n,\e} - f^{n-1,\e}_t\|_{\mathcal{C}^0}^2 + \|\varphi_t^{n,\e} - \varphi_t^{n-1,\e}\|_{\mathcal{C}^0}^2\\
&\hspace{0.4cm} \le \mathcal{B}_t^1  \int_0^t  \mathcal{C}_s^1 \Big( \|\varphi_s^{n,\e} - \varphi_s^{n-1,\e}\|_{\mathcal{C}^0}^2+ \|f_s^{n,\e} - f_s^{n-1,\e}\|_{\mathcal{C}^0}^2 \Big )ds \\
& \hspace{0.6cm} + (1+2\|f^{in}\|_{W^{1,\infty}}\mathcal{A}_t^1) \mathcal{B}_t^2  \left( \ \int_0^t  \mathcal{C}_s^2( \|\varphi_s^{n,\e} - \varphi_s^{n-1,\e}\|_{\mathcal{C}^0}^2+ \|f_s^{n-1,\e} - f_s^{n-2,\e}\|_{\mathcal{C}^0}^2)ds  \right)\\
& \hspace{0.4cm} \le \tilde{\mathcal{D}}_t \int_0^t  \Big( \|\varphi_s^{n,\e} - \varphi_s^{n-1,\e}\|_{\mathcal{C}^0}^2+ \|f_s^{n-1,\e} - f_s^{n-2,\e}\|_{\mathcal{C}^0}^2\Big)ds,  
\end{align*}
where

\begin{align*}
&\mathcal{B}_t^1 := 6\left[1+ t \left(d\|\phi\|_{W^{1,\infty}}\|f^{in}\|_{L^\infty} \exp(d\phi_Mt + d\sigma W_t)\right)^2\right],\\
&\mathcal{C}_t^1 := (1 + (4\mathcal{X}^\infty(t)\mathcal{V}^\infty(t))^{2d}),\\
&\mathcal{B}_t^2 := 1+2\|\phi\|_{W^{1,\infty}} \exp\left(4\sigma \sup_{0 \le s \le t} |W_s|\right),\\
&\mathcal{C}_t^2 := 1 + \mathcal{V}^\infty(t)(4\mathcal{X}^\infty(t) \mathcal{V}^\infty(t))^d,\\
&\tilde{\mathcal{D}}_t := \mathcal{B}_t^1 \left(\sup_{0 \le s \le t}\mathcal{C}_s^1\right) + \left(1+2\|f^{in}\|_{W^{1,\infty}}\mathcal{A}_t^1)\right) \mathcal{B}_t^2 \left(\sup_{0 \le s \le t}\mathcal{C}_s^2\right).
\end{align*}
This gives the desired result.

\end{proof}
For each $t$ and $\omega\in \Omega$, we define
\[ \Delta_{n}^\varepsilon(t,\omega) := \|f_t^{n,\varepsilon} - f_t^{n-1,\varepsilon} \|_{\mathcal{C}^0}^2 + \|\varphi_t^{n,\varepsilon} - \varphi_t^{n-1,\varepsilon} \|_{\mathcal{C}^0}^2 . \]
\begin{corollary} \label{C3.2} The functional $\Delta_{n}^\varepsilon(t)$ satisfies
\[\Delta_n^\varepsilon(t,\omega) \le \frac{(\mathcal{K}(\omega) t)^n}{n!}, \quad \mbox{for each} \ \  t \in [0,T] \quad \mbox{and a.s.  } \omega \in \Omega,  \]
where $\mathcal{K} = \mathcal{K}(\omega)$ is a nonnegative random variable.
\end{corollary}
\begin{proof}

It follows from Proposition \ref{P3.4} that 
\[\Delta_{n+1}^\varepsilon(t) \le \tilde{\mathcal{D}}_t \left( \int_0^t (\Delta_n^\varepsilon(s) + \Delta_{n+1}^\varepsilon(s)) ds \right). \]
Since $\tilde{\mathcal{D}}_t$ is a nonnegative process with continuous sample paths, there exists a nonnegative random variable $D = D(\omega)$ such that
\[\sup_{0 \le t \le T} \tilde{\mathcal{D}}_t(\omega) \le D(\omega)<\infty, \quad \mbox{for each} \quad \omega \in \Omega. \]
Thus, we can use the Gr\"onwall-type lemma in Lemma \ref{L4.1} to deduce
\[\Delta_n(t,\omega) \le \frac{(\mathcal{K}(\omega)t)^n}{n!}, \quad \mbox{for each} \ \  t \in [0,T], \ \omega \in \Omega,  \]
where $\mathcal{K} = \mathcal{K}(\omega)$ depends on $D(\omega)$.
\end{proof}

\begin{remark}\label{R3.3}
Corollary \ref{C3.2} implies that for every $\omega$, 
\[ f_t^{n,\varepsilon}(\omega) \to f_t^\varepsilon(\omega) \quad \mbox{in}~\mathcal{C}([0,T] \times \bbr^{2d}). \]
Since $f_t^{n,\varepsilon}$ is $\mathcal{F}_t$-adapted (where $\mathcal{F}_t$ is a filtration generated by the Wiener process) and $f_t^\e$ is a pointwise limit of $f^{n,\varepsilon}$ over $\Omega$, we have $f$ is $\mathcal{F}_t$-adapted. Moreover, we have a uniform boundedness of $f_t^{n,\varepsilon}$ in $L^\infty([0,T];W^{m,p}( \bbr^{2d}))$ for any $p \in [1,\infty)$.  By the property of reflexive Banach space, there exists a subsequence $\{f^{n_k,\varepsilon}(\omega)\} \subseteq \{f^{n,\varepsilon}(\omega)\}$ which is weakly convergent to $\tilde{f}_t(\omega)$ in $L^\infty([0,T];W^{m,p}( \bbr^{2d}))$ for each $\omega \in \Omega$ and every $p \in [1,\infty)$. Since we have already a strong convergence in the lower order, we can conclude that $f_t^\e(\omega) = \tilde{f}_t(\omega)$.  However, we can not proceed further, since it is not clear whether $f_t^\e$ satisfies the equation \eqref{A-2} at this moment. This is due to the noise term in the right-hand side of \eqref{C-1}. It is not certain whether the Stratonovich integral of $f_t^\e$ can be defined or not. In addition, even if the noise term can be well-defined, it is also not clear whether the Stratonovich integral of $f^{n,\varepsilon}$ converges to that of $f_t^\e$ or not.
 \end{remark}
\subsection{Proof of Theorem \ref{T2.3}} 
In this subsection, we prove a global existence of a solution to system \eqref{C-1} by showing that the limit of the sequence $\{f_t^{n,\varepsilon}\}$ exists as $n \to \infty$ for each $\varepsilon$, and that this limit is indeed a strong solution to \eqref{B-1} corresponding to the regularized initial datum $f^{in,\varepsilon}$. 

In order to cope with the problems discussed in Remark \ref{R3.3}, we employ a stopping time argument. First, we define a sequence of stopping times 
$\{\tau_M\}_{M \in \bbn}$ as follows:
\begin{align*}
\begin{aligned}
&\tau_M^1(\omega) := \inf\{t\ge 0 \ | \ \mathcal{A}^{k_*}_t(\omega) >M \} \wedge T, \quad \tau_M^2(\omega) := \inf\{t \ge 0 \ | \ \tilde{\mathcal{D}}_t(\omega) > M\} \wedge T,\\
&\tau_M^2(\omega) := \inf\{t \ge 0 \ | \ \mathscr{D}_t(\omega) > M\} \wedge T, \quad \tau_M := \tau_M^1 \wedge \tau_M^2\wedge \tau_M^3,
\end{aligned}
\end{align*}
where $k_* := \max\{k,4\}$ and $\mathscr{D}_t$ is a nonnegative process with continuous sample paths which will be specified later. Now, we verify the existence of regularized solutions step by step. \newline

\noindent $\bullet$ (Step A: The limit $n \to \infty$): First, we obtain the limit function $f_{t\wedge\tau_M}^\e$ which is a classical solution to equation \eqref{A-2} with the regularized initial data, based on the estimates in the previous subsection.\newline

\noindent $\diamond$ (Step A-1: Extracting a limit function): We can find out that for each $n \in \bbn$,
\begin{align*}
&(i)~~ \|f^{n,\varepsilon}_{t \wedge \tau_M}\|_{W^{m,\infty}} \le  M \|f^{in,\varepsilon}\|_{W^{m,\infty}}. \\
&(ii)~~ \|f^{n,\varepsilon}_{t \wedge \tau_M} - f^{n-1,\varepsilon}_{t \wedge \tau_M}\|_{\mathcal{C}^0}^2 + \|\varphi_{t \wedge \tau_M}^{n,\varepsilon} - \varphi_{t \wedge \tau_M}^{n-1,\varepsilon}\|_{\mathcal{C}^0}^2\\
& \hspace{1cm}\le M \Big[ \int_0^t   \Big( \|\varphi_{s \wedge \tau_M}^{n,\varepsilon} - \varphi_{s \wedge \tau_M}^{n-1,\varepsilon}\|_{\mathcal{C}^0}^2+ \|f^{n-1,\varepsilon}_{s \wedge \tau_M} - f^{n-2,\varepsilon}_{s \wedge \tau_M}\|_{\mathcal{C}^0}^2 \Big) ds \Big].
\end{align*}
Thus, we can use the same argument as in Corolllary \ref{C3.2} to yield that as $n \to \infty$, there exists a limit function $ f_{t \wedge \tau_M}^\varepsilon$ such that, up to a subsequence, 
\begin{align*}
&f^{n,\varepsilon}_{t \wedge \tau_M} \to f_{t \wedge \tau_M}^\varepsilon \quad \mbox{in} \quad L^\infty(\Omega ;\mathcal{C}( [0,T]\times\bbr^{2d})),\\
&f^{n,\varepsilon}_{t \wedge \tau_M} \rightharpoonup f_{t \wedge \tau_M}^\varepsilon \quad \mbox{in} \quad L^\infty(\Omega \times [0,T]; W^{m,p}(\bbr^{2d})), \quad \forall \ p \in [1,\infty).
\end{align*}

\vspace{0.4cm}

\noindent $\diamond$ (Step A-2: Verification of relation \eqref{B-0-1}): Now, we need to show that $f_{t \wedge \tau_M}^\varepsilon$ satisfies \eqref{C-1} in the sense of Definition \ref{D2.1}. Since $f_{t\wedge\tau_M}^{n,\varepsilon}$ satisfies \eqref{C-1-1} and conditions of Lemma \ref{L2.2}, it satisfies the following relation:

\begin{align}
\begin{aligned} \label{C-9}
&\int_\Sigma f^{n,\varepsilon}_{t \wedge \tau_M} \psi dz = \int_\Sigma f^{in,\varepsilon} \psi dz + \int_0^t \int_\Sigma f^{n,\varepsilon}_{s\wedge \tau_M} \left( v \cdot \nabla_x \psi + \left(F_a[f^{n-1,\varepsilon}_{s \wedge \tau_M}] + \frac{1}{2}\sigma^2 v \right)\cdot \nabla_v \psi \right) dzds\\
& \hspace{0.7cm}+\frac{1}{2}\sigma^2 \int_0^t \int_\Sigma v f^{n,\varepsilon}_{s \wedge \tau_M}  \cdot (D_v^2 \psi)v dzds  - \sigma\int_0^t \int_\Sigma f^{n,\varepsilon}_{s\wedge \tau_M} v \cdot \nabla_v \psi dz dW_s, 
\end{aligned}
\end{align}
where $\Sigma := \bbr^{2d}$ and $dz = dvdx$. \newline

Next, our job is to pass $n \to \infty$ in the integral relation \eqref{C-9} to derive an integral relation \eqref{B-0-1} for $f^\varepsilon_{t \wedge \tau_M}$. For this, note that the $x$- and $v$-supports of $f^{n,\varepsilon}_{t \wedge \tau_M}$ and $f^\varepsilon_{t \wedge \tau_M}$ are uniformly bounded by $|\mathcal{X}_{t \wedge \tau_M}^\infty|$ and $|\mathcal{V}_{t \wedge \tau_M}^\infty|$ (see Corollary \ref{C3.1}). Moreover, we can find out that $|\mathcal{X}_{t \wedge \tau_M}^\infty|$ and $|\mathcal{V}_{t \wedge \tau_M}^\infty|$ are bounded by $\mathcal{A}_{t \wedge \tau_M}^{k_*}$, and hence by $M$. We combine the strong convergence on the lower order with these facts to yield
\begin{align*}
&(i)~~ \int_\Sigma (f^{n,\varepsilon}_{t \wedge \tau_M}-f_{t \wedge \tau_M}^\varepsilon) \psi dz \longrightarrow 0. \\
&(ii)~~  \int_0^t \int_\Sigma (f^{n,\varepsilon}_{s\wedge \tau_M}-f_{s \wedge \tau_M}^\varepsilon) \left( v \cdot \nabla_x \psi + \left(F_a[f^{n-1,\varepsilon}_{s \wedge \tau_M}] + \frac{1}{2}\sigma^2 v \right)\cdot \nabla_v \psi \right) dzds\longrightarrow 0. \\
&(iii)~~ \int_0^t \int_\Sigma f_{s \wedge \tau_M}^\varepsilon \left( F_a [ f_{s\wedge\tau_M}^{n-1,\varepsilon}] - F_a [f^{\varepsilon}_{s\wedge\tau_M}] \right) \nabla_v \psi dzds \longrightarrow 0. \\
&(iv)~~ \frac{1}{2}\sigma^2 \int_0^t \int_\Sigma v(f^{n,\varepsilon}_{s \wedge \tau_M}-f_{s \wedge \tau_M}^\varepsilon)\cdot (D_v^2 \psi)v dzds \longrightarrow 0,
\end{align*}
uniformly in $\omega$, as $n$ goes to infinity.  \newline

Now it remains to check with the stochastic integral term in \eqref{C-9}. For this term, one has
\begin{align*}
&\bbe \left[\left(\int_0^t \int_\Sigma (f^{n,\varepsilon}_{s \wedge \tau_M} - f_{s \wedge \tau_M}^\varepsilon) v \cdot \nabla_v \psi dz dW_s\right)^2 \right] = \bbe \left[\int_0^t \left(\int_\Sigma (f^{n}_{s\wedge \tau_M} - f_{s\wedge \tau_M}^\varepsilon) v \cdot \nabla_v \psi dz \right)^2 ds \right]\\
& \hspace{1cm} \le \bbe\left[\int_0^t \|f^{n,\varepsilon}_{s \wedge \tau_M} - f_{s \wedge \tau_M}^\varepsilon\|_{\mathcal{C}^0}^2 ds \right] \left(\int_\Sigma |v \cdot \nabla_v \psi|dz \right)^2 \longrightarrow 0, \quad \mbox{as}\quad n \to \infty.
\end{align*}

\noindent This $L^2$-convergence over $\Omega$ implies that there exists a subsequence $\{f^{n_l,\varepsilon}_{t \wedge \tau_M}\}$ such that
\[\left(\int_0^t \int_\Sigma f^{n_l,\varepsilon}_{s \wedge \tau_M} v \cdot \nabla_v \psi dz dW_s\right)(\omega) \longrightarrow \left(\int_0^t \int_\Sigma f_{s \wedge \tau_M}^\varepsilon v \cdot \nabla_v \psi dz dW_s\right)(\omega),  \]
for a.s. $\omega$, as $l$ goes to infinity. Thus, we can conclude that for a.s. $\omega \in \Omega$, $f_{s \wedge \tau_M}^\varepsilon$ satisfies
\begin{align*}
\int_\Sigma f_{t \wedge \tau_M}^\varepsilon \psi dz &= \int_\Sigma f^{in,\varepsilon} \psi dz - \int_0^t \int_\Sigma f_{s\wedge \tau_M}^\varepsilon \left( v \cdot \nabla_x \psi + \left(F_a[f^\varepsilon_{s \wedge \tau_M}] + \frac{1}{2}\sigma^2 v \right)\cdot \nabla_v \psi \right) dzds\\
&-\frac{1}{2}\sigma^2 \int_0^t \int_\Sigma v f^\varepsilon_{s \wedge \tau_M}\cdot (D_v^2 \psi)v dzds  + \int_0^t \int_\Sigma f^\varepsilon_{s\wedge \tau_M} v \cdot \nabla_v \psi dz dW_s,
\end{align*}
for every $\psi \in \mathcal{D}( \bbr^{2d})$. One also has $f_{t\wedge\tau_M}^\varepsilon$ is a $\mathcal{F}_t$-semimartingale. Here, we use Lemma \ref{L3.1} to obtain that $f_{t\wedge\tau_M}^\varepsilon$ satisfies \eqref{A-2} in the sense of distribution. \newline

\noindent $\bullet$ (Step B: The limit $\e \to 0$):~Here, we address the convergence of solutions to the regularized system \eqref{C-0}. Since $k^* \ge 4$, one uses the Sobolev embedding theorem to get $f_{t\wedge\tau_M}^\varepsilon \in L^\infty(\Omega; \mathcal{C}([0,T];\mathcal{C}^{3,\delta}(\bbr^{2d})))$. Thus, it follows from Lemma \ref{L2.2} and Remark 3.1 that $f_{t\wedge\tau_M}^\varepsilon$ becomes a classical solution to \eqref{C-0} corresponding to the regularized initial datum $f^{in, \e}$ .\newline

\noindent $\diamond$ (Step B-1: Extracting a limit function):~ Note that the strong convergence in Step A implies that the $x$-support and the $v$-support of $f_{t\wedge\tau_M}^\e$ are bounded by $\mathcal{X}^\infty$ and $\mathcal{V}^\infty$, respectively, uniformly in $\e$. Thus, we can follow the stability estimate in Theorem \ref{T3.1} to get
\begin{align}
\begin{aligned}  \label{D-1}
&\|f_{t\wedge \tau_M}^\varepsilon-f_{t\wedge\tau_M}^{\varepsilon'} \|_{\mathcal{C}^0}^2 + \| \varphi_{t\wedge \tau_M}^\varepsilon-\varphi_{t\wedge\tau_M}^{\varepsilon'} \|_{\mathcal{C}^0}^2  \\
& \hspace{2cm} \le \mathscr{D}_{t\wedge\tau_M} \| f^{in, \varepsilon} - f^{in, \varepsilon'}\|_{\mathcal{C}^0}^2 \le M \|f^{in, \varepsilon} - f^{in, \varepsilon'} \|_{\mathcal{C}^0}^2,
\end{aligned}
\end{align}
where $\mathscr{D}_t$ can be obtained if $\mathcal{R}(t)$, $\mathcal{P}(t)$,  $\max(\|f^{in}\|_{L^\infty}, \|\tilde{f}^{in}\|_{L^\infty})$ and $\max(\|f_t\|_{W^{1,\infty}}, \|\tilde{f}_t\|_{W^{1,\infty}})$ in the formulation of $\mathcal{D}_t$ from Theorem \ref{T3.1} are substituted by $\mathcal{X}^\infty(t)$, $\mathcal{V}^\infty(t)$, $\|f^{in}\|$ and $\|f^{in}\|_{W^{1,\infty}} \mathcal{A}_t^1$, respectively.\newline

Since $f^{in, \varepsilon}$ converges uniformly to $f^{in}$, it follows from the stability estimate \eqref{D-1} that there exists  $f_{t\wedge\tau_M}$ such that 
\[ 
f^\varepsilon_{t \wedge \tau_M} \to f_{t \wedge \tau_M} \quad \mbox{in} \quad L^\infty(\Omega;\mathcal{C}( [0,T]\times\bbr^{2d})). 
\]

\noindent Moreover, it follows from the weak convergence and ($\mathcal{F}^\e 1$) that

\[
\|f^{\varepsilon}_{t \wedge \tau_M}\|_{W^{k,\infty}} \le \mathcal{A}_t^k \|f^{in, \varepsilon}\|_{W^{k,\infty}} \le  M \|f^{in}\|_{W^{k,\infty}}. 
\]

\noindent Hence, we can follow the arguments in Step A to yield that $f_{t\wedge\tau_M}$ satisfies relation \eqref{B-0-2} and hence \eqref{B-0-1}. Moreover, $f_{t\wedge\tau_M}$ is compactly supported in $x$ and $v$. \newline

\noindent $\diamond$~(Step B-2: Regularity of a strong solution): Now, we prove that $f_{t\wedge\tau_M}$ has the desired regularity. Since $f_{t\wedge\tau_M}^\e$ is a classical solution to \eqref{A-2} with initial data $f^{in,\e}$, it can be uniquely written as 
\[ f_{t\wedge\tau_M}^\e(\varphi_{t\wedge\tau_M}^\e(x,v)) = f^{in,\e}(x,v)\exp\left[ -\int_0^{t\wedge\tau_M} \nabla_v \cdot F_a[f_s](\varphi_s^\e)ds + d\sigma  W_{t\wedge\tau_M}\right],  \] 
(for detail, we refer to Appendix B). Since we also obtain the uniform convergence of the characteristics $\varphi_{t\wedge\tau_M}^\e$ as $\e \to 0$ from \eqref{D-1}, the solution $f_{t\wedge\tau_M}$ satisfies the following relation:
\begin{equation}\label{D-2} 
f_{t\wedge\tau_M}(\varphi_{t\wedge\tau_M}(x,v)) = f^{in}(x,v)\exp\left[ -\int_0^{t\wedge\tau_M} \nabla_v \cdot F_a[f_s](\varphi_s)ds + d\sigma  W_{t\wedge\tau_M}\right],  
\end{equation}
and the limit $\varphi_{t\wedge\tau_M}(x,v) = (X_{t\wedge\tau_M}(x,v), V_{t\wedge\tau_M})$ is a solution to the following SDE:
\[
\begin{cases}
\displaystyle X_{t\wedge\tau_M} = x+  \int_0^{t\wedge\tau_M} V_s ds,\\
\displaystyle V_{t\wedge\tau_M} = v + \int_0^{t\wedge\tau_M}\left(F_a[f_s](X_s, V_s) \right)ds + \int_0^{t\wedge\tau_M}\sigma (v_c-V_t) \circ dW_s.
\end{cases}
\]
Since the kernel $F_a[f_t]$ is smooth, $\varphi_{t\wedge\tau_M}(x,v)$ can be shown to be a $\mathcal{C}^m$-diffeomorphism for any $m\in\bbn$, and so is its inverse $\psi_{t\wedge\tau_M}(x,v) :=( \varphi_t(x,v))^{-1}$. Thus, if we write
\[ 
f_{t\wedge\tau_M}(x,v) = f^{in}(\psi_{t\wedge\tau_M}(x,v))\exp\left[ -\int_0^{t\wedge\tau_M} \nabla_v \cdot F_a[f_s](\varphi_s(\psi_{t\wedge\tau_M}(x,v))ds + d\sigma  W_{t\wedge\tau_M}\right],  
\]
it directly follows from the regularity of $f^{in}$ and $\psi_t$ that $f_{t\wedge\tau_M}$ has the desired regularity. \newline

\noindent $\bullet$~(Step C: Properties of a strong solution):~ We recall several properties of regularized solutions. First, it is obvious from \eqref{D-2} that

\[
\|f_{t\wedge\tau_M}\|_{L^\infty} \le \|f^{in}\|_{L^\infty} \exp(d\phi_M t\wedge\tau_M + d\sigma W_{t\wedge\tau_M})
\]
Since $f_{t\wedge\tau_M}^\varepsilon$ is a classical solution to \eqref{C-0} corresponding to the regularized initial datum $f^{in, \e}$, Lemma \ref{L2.3} gives

\[ 
M_2^\varepsilon (t\wedge\tau_M) \le M_2^\e(0) \exp(-2\phi_m t \wedge\tau_M - 2\sigma W_{t\wedge\tau_M}),
\]
and the strong convergence together with compact supports gives

\[ 
M_2 (t\wedge\tau_M) \le M_2(0) \exp(-2\phi_m t \wedge\tau_M - 2\sigma W_{t\wedge\tau_M}).
\]

\noindent Moreover, it is obvious that
\[ \tau_M(\omega) \to T \quad \mbox{as $M \to \infty$ for a.s. $\omega$}. \] 
Thus, we choose a sufficiently large $M$ for each $\omega \in \Omega$ such that $f_{t\wedge\tau_M }(\omega)$ satisfies the relation \eqref{B-0-1} on $[0,T]$.\newline

\noindent For the expectation estimates of the solution, we use Fatou's lemma to get, for any $p \in (1,\infty)$,
\begin{align*}
\begin{aligned}
&\mathbb{E} \|f_t\|_{L^\infty} \le \liminf_{M \to \infty} \mathbb{E}\|f_{t\wedge\tau_M}\|_{L^\infty} \le \liminf_{M \to \infty} \|f^{in}\|_{L^\infty} \mathbb{E}\Big[ \exp(d\phi_M t\wedge\tau_M + d\sigma W_{t\wedge\tau_M}) \Big]\\
&  \hspace{0.4cm} = \liminf_{M \to \infty}\|f^{in}\|_{L^\infty} \mathbb{E}\Big[ \exp\left( d\sigma W_{t \wedge \tau_M} - \frac{p(d\sigma)^2}{2}t \wedge \tau_M \right)\exp \left(\Big(d\phi_M+ \frac{p(d\sigma)^2}{2}\Big) t\wedge\tau_M \right) \Big]\\
& \hspace{0.4cm} \le \liminf_{M \to \infty}\|f^{in}\|_{L^\infty} \mathbb{E} \Big[\exp \left(\frac{p}{p-1}\Big(d\phi_M+ \frac{p(d\sigma)^2}{2}\Big) t\wedge\tau_M \right)  \Big]^{(p-1)/p}\\
& \hspace{0.4cm} =\|f^{in}\|_{L^\infty} \exp\left(\Big(d\phi_M+ \frac{p(d\sigma)^2}{2}\Big) t \right),
\end{aligned}
\end{align*}
where we used the fact $X_t = \exp(aW_t - a^2 t /2) $ is a martingale, H\"older's inequality and Lebesgue's dominated convergence theorem. Then we take the limit $ p \to 1$ on both sides to obtain the desired result. For the dissipation of the second velocity moment, we use a similar argument to get the desired estimate.

\section{Conclusion}\label{sec:5}
In this paper, we presented a global existence of strong solutions and their asymptotic emergent dynamics for the stochastic kinetic Cucker-Smale equation perturbed by a multiplicative white noise. For a global well-posedness, we first derive a sequence of classical solutions to the stochastic kinetic C-S equation with regularized initial data. Then, using the properties of classical solutions, we obtained the existence of a strong solution  corresponding to the original initial data and asymptotic emergent stochastic dynamics of strong solutions. Of course, there are lots of interesting issues to be addressed in a future work, e.g., a global existence of weak solutions, emergent dynamics under other types of random perturbations and zero noise limit, etc. These topics will be discussed in future works.

\newpage

\appendix

\section{Elementary lemmas}
\setcounter{equation}{0}
In this appendix, we provide two useful lemmas used in previous sections. First, we begin with estimate on a variant of geometric brownian motion.
\begin{lemma}\label{LA-1}
Let $\{X_t\}_{t \ge 0}$ be a solution satisfying the following Cauchy problem:
\[
\begin{cases}
\displaystyle dX_t = (a_t + b_t X_t) dt + c X_t dW_t, \quad  t > 0, \\
\displaystyle X_0 = x \ge 0,
\end{cases}
\]
where $\{a_t\}_{t \ge 0}$ and $\{b_t\}_{t \ge 0}$ are stochastic processes with continuous sample paths, and $c$ is a constant. Then one has
\[
X_t = x \exp\Big[ \int_0^t \Big(b_s - \frac{c^2}{2} \Big) ds + cW_t \Big]  + \int_0^t a_s \exp\Big[ \int_s^t \Big(b_\tau - \frac{c^2}{2} \Big) d\tau + c(W_t - W_s) \Big] ds.
\]
\end{lemma}
\begin{proof} 
The proof is exactly given in Example 19.7 from \cite{S-P}. So, we refer to \cite{S-P} for its proof.

\end{proof}
\begin{lemma}\label{LA-2}
\emph{(Comparision principle)}
Suppose that two stochastic proceeses $\{X_t\}_{t \ge 0}$ and $\{Y_t\}_{t \ge 0}$ satisfy 
\begin{align*}
&dX_t \le (a_t + bX_t)dt + cX_t dW_t, \quad X_0 = x \ge 0,\\
&dY_t  = (a_t + bY_t)dt + cY_t dW_t, \quad Y_0 = x,
\end{align*}
where $\{a_t\}_{t \ge 0}$ is a stochastic process with continuous sample paths. Then, we have
\[X_t \le Y_t, \quad \forall t \ge0. \]
\end{lemma}
\begin{proof}
Let $\{Y_t^\delta\}_{t \ge 0}$, ($\delta>0$) be a stochastic process satisfying
\[
\begin{cases}
dY_t^\delta  = (a_t + bY_t^\delta)dt + cY_t^\delta dW_t,~~t > 0, \\
Y_0^\delta = x+\delta,
\end{cases}
\]
and we set
\[ Z_t^\delta := Y_t^\delta -X_t. \]
Then, we have
\[dZ_t^\delta \ge bZ_t^\delta dt + cZ_t^\delta dW_t, \quad t > 0 \quad \mbox{and} \quad Z_0 = \delta, \quad t = 0.\]

\noindent We use It\^o's lemma to get

\[d(\ln Z_t^\delta) = \frac{dZ_t^\delta}{Z_t^\delta} - \frac{1}{2} \frac{1}{(Z_t^\delta)^2} (dZ_t^\delta) \cdot (dZ_t^\delta) \ge \left(b_t-\frac{c^2}{2} \right) dt + c~dW_t.  \]
Again, we integrate the above relation to get
\[Z_t^\delta \ge \delta  \exp \left\{ \int_0^t \left(b_s-\frac{c^2}{2} \right) ds + cW_t \right\} \geq 0. \]
This yields
\[ X_t \le Y_t^\delta \quad \mbox{for all $ t \ge 0$}. \]
It follows from the representation formula in Lemma \ref{LA-1} that
\begin{align*}
\begin{aligned}
Y_t^\delta &= (x+\delta) \exp\left\{  \int_0^t \left(b_s-\frac{c^2}{2} \right) ds + cW_t\right\} + \int_0^t a_s \exp\Big[ \int_s^t \Big(b_\tau - \frac{c^2}{2} \Big) d\tau + c(W_t - W_s) \Big] ds, \\
Y_t &= x \exp\left\{  \int_0^t \left(b_s-\frac{c^2}{2} \right) ds + cW_t\right\} + \int_0^t a_s \exp\Big[ \int_s^t \Big(b_\tau - \frac{c^2}{2} \Big) d\tau + c(W_t - W_s) \Big] ds.
\end{aligned}
\end{align*}
This yields the desired result:
\[Y_t = \liminf_{\delta \to 0} Y_t^\delta \ge X_t. \]
\end{proof}

\section{A proof of Theorem \ref{T3.1}}
\setcounter{equation}{0}

First, we define $\eta$ and $\tilde\eta$ as follows:
\begin{align*}
\eta_t(x,v) &:= f^{in}(x,v) - \int_0^t \eta_s(x,v)(\nabla_v \cdot F_a[f_s])(\varphi_s)ds + \sigma d\int_0^t\eta_s(x,v)  \circ dW_s,\\
\tilde\eta_t(x,v) &:= \eta_t ((\varphi_t)^{-1}).
\end{align*}
\noindent We use the generalized It\^o's formula from Theorem 3.3.2 in \cite{H.K} to obtain that $\tilde\eta_t$ satisfies the relation \eqref{A-2}. Since the classical solutions can become measure-valued solutions and the uniqueness of measure-valued solutions is guaranteed in Theorem \ref{T2.1}, we have 
\[  \tilde\eta_t(x,v) =  f_t(x,v). \]
Moreover, since $\eta$ is a geometric Brownian motion, a unique classical solution $f$ corresponding to the initial datum $f^{in}$ can be represented by
\[\eta_t(x,v) = f_t(\varphi_t(x,v)) = f^{in}(x,v)\exp\left[ -\int_0^t \nabla_v \cdot F_a[f_s](\varphi_s)ds + d\sigma  W_t\right].  \]
Now, we consider another classical solution $\tilde{f}$ corresponding to the initial datum $\tilde{f}^{in}$ and the associated stochastic flow $\tilde{\varphi}_t(x,v)$. Moreover, we set
\begin{align*}
&\mathcal{R}(t) := \sup \left\{ |x| \ : \ f_t(x,v) \neq 0 \quad \mbox{or} \quad \tilde{f}_t(x,v) \neq 0 \quad \mbox{for some } \ v \in \bbr^d \right\},\\
&\mathcal{P}(t) := \sup \left\{ |v| \ : \ f_t(x,v) \neq 0 \quad \mbox{or} \quad \tilde{f}_t(x,v) \neq 0 \quad \mbox{for some } \ x \in \bbr^d \right\}.
\end{align*}

Then, we claim

\begin{align}
\begin{aligned} \label{T3-1.1}
&(i)~~ \|f_t - \tilde{f}_t\|_{L^{\infty}}^2 \le \mathcal{B}_t^1 \Big[  \|f^{in} - \tilde{f}^{in}\|_{L^\infty}^2 +  \int_0^t  \mathcal{C}_s^1 \Big( \|\varphi_s - \tilde{\varphi}_s\|_{L^\infty}^2+ \|f_s - \tilde{f}_s\|_{L^{\infty}}^2 \Big )ds  \Big] \\
&\hspace{3cm} + 2\max(\|f_t\|_{W^{1,\infty}}, \|\tilde{f}_t\|_{W^{1,\infty}}) \|\varphi_t - \tilde{\varphi}_t\|_{L^\infty}^2,  \\
&(ii)~~\|\varphi_t - \tilde{\varphi}_t\|_{L^\infty}^2\le \mathcal{B}_t^2  \left( \ \int_0^t  \mathcal{C}_s^2( \|\varphi_s - \tilde{\varphi}_s\|_{L^\infty}^2+ \|f_s - \tilde{f}_s\|_{L^{\infty}}^2)ds  \right),
\end{aligned}
\end{align}
where $\mathcal{B}_t^i$ and $\mathcal{C}_t^i$ ($i=1,2$) are nonnegative processes which have continuous sample paths.\newline

\noindent (i)~~ First, we derive the $L^\infty$-estimates for classical solutions:
\begin{align*}
\begin{aligned}
&f(\varphi_t(x,v)) - \tilde{f}_t(\varphi(x,v)) = \left(f(\varphi_t(x,v)) -\tilde{f}(\tilde{\varphi}_t(x,v))\right) + \left(\tilde{f}(\tilde{\varphi}_t(x,v)) - \tilde{f}(\varphi_t(x,v))\right)\\
& \hspace{1cm} =: \mathcal{I}_{21} + \mathcal{I}_{22}.
\end{aligned}
\end{align*}

\vspace{0.2cm}

\noindent $\bullet$~(Estimate of $\mathcal{I}_{21}$):~By direct estimate, one has
\begin{align*}
\begin{aligned}
\mathcal{I}_{21} &=f^{in}(x,v)\exp\left[ -\int_0^t \nabla_v \cdot F_a[f_s](\varphi_s)ds + d\sigma  W_t\right]  -\tilde{f}^{in}(x,v)\exp\left[ -\int_0^t \nabla_v \cdot F_a[\tilde{f}_s](\varphi_s)ds + d\sigma  W_t\right]\\
&\le \|f^{in} - \tilde{f}^{in}\|_{L^\infty}\exp\left[ -\int_0^t \nabla_v \cdot F_a[f_s](\varphi_s)ds + d\sigma  W_t\right]\\
&\quad + \|\tilde{f}^{in}\|_{L^\infty}\exp(d\sigma W_t) \left[ \exp\bigg( -\int_0^t \nabla_v \cdot F_a[f_s](\varphi_s)ds\bigg) - \exp\bigg( -\int_0^t \nabla_v \cdot F_a[\tilde{f}_s](\tilde{\varphi}_s)ds\bigg)\right]\\
&\le \|f^{in} - \tilde{f}^{in}\|_{L^\infty}\exp\left( d\phi_M t + d\sigma  W_t\right)\\
&\quad + \|\tilde{f}^{in}\|_{L^\infty}\exp(d\phi_M t + d\sigma W_t) \left| \int_0^t \left(\nabla_v \cdot F_a[f_s](\varphi_s) - \nabla_v \cdot F_a[\tilde{f}_s](\tilde{\varphi}_s)\right)ds\right|,
\end{aligned}
\end{align*}
where we used the mean-value theorem, and we have
\begin{align*}
\begin{aligned}
&\left|\nabla_v \cdot F_a[f_s](\varphi_s) - \nabla_v \cdot F_a[\tilde{f}_s](\tilde{\varphi}_s)\right|\\
& \hspace{0.5cm} \le d \int_{\bbr^{2d}}\left|\phi(x_* - X_s ) - \phi(x_* - \tilde{X}_s)\right| f_s dv_*dx_* + d\int_{\bbr^{2d}}\phi(x_* - \tilde{X}_t) |f_s -\tilde{f}_s|dv_*dx_*\\
& \hspace{0.5cm} \le d\phi_{Lip}|X_s - \tilde{X}_s| + d\phi_M (4\mathcal{R}(s)\mathcal{P}(s))^d\|f_s - \tilde{f}_s\|_{L^\infty}.
\end{aligned}
\end{align*}
Thus, we get
\begin{align*}
\mathcal{I}_{21} &\le \|f^{in} - \tilde{f}^{in}\|_{L^\infty}\exp\left( d\phi_M t + d\sigma  W_t\right) + \|\tilde{f}^{in}\|_{L^\infty}\exp(d\phi_M t + d\sigma W_t)  \int_0^t d\phi_{Lip}|X_s - \tilde{X}_s|ds\\
&+\|\tilde{f}^{in}\|_{L^\infty}\exp(d\phi_M t + d\sigma W_t) \int_0^t d\phi_M (4\mathcal{R}(s)\mathcal{P}(s))^d\|f_s - \tilde{f}_s\|_{L^\infty} ds.
\end{align*}

\noindent $\bullet$~(Estimate of $\mathcal{I}_{22}$): By direct estimate, one has
\[ \mathcal{I}_{22} \le \|\tilde{f}_t\|_{W^{1,\infty}}\|\varphi_t -\tilde{\varphi}_t\|_{L^\infty}.\]
Hence, we take the supremum over $(x,v)\in\bbr^d \times \bbr^d$, and use Young's inequality and the Cauchy-Schwarz inequality to get
\begin{align*}
\|f_t - \tilde{f}_t\|_{L^\infty} &\le 2\mathcal{I}_{21}^2 + 2\mathcal{I}_{22}^2 \le 6\|f^{in} - \tilde{f}^{in}\|_{L^\infty}^2\exp\left( 2d\phi_M t + 2d\sigma  W_t\right)\\
&+ 6\|\tilde{f}^{in}\|_{L^\infty}^2\exp\left( 2d\phi_M t + 2d\sigma  W_t\right)\left(\int_0^t  d\phi_{Lip}|X_s - \tilde{X}_s| ds\right)^2\\
&+ 6\|\tilde{f}^{in}\|_{L^\infty}^2\exp\left( 2d\phi_M t + 2d\sigma  W_t\right) \left(\int_0^t d\phi_M (4\mathcal{R}(s)\mathcal{P}(s))^d\|f_s - \tilde{f}_s\|_{L^\infty} ds\right)^2\\
&+2\|\tilde{f}_t\|_{W^{1,\infty}}^2\|\varphi_t -\tilde{\varphi}_t\|_{L^\infty}^2\\
&\le  6\|f^{in} - \tilde{f}^{in}\|_{L^\infty}^2\exp\left( 2d\phi_M t + 2d\sigma  W_t\right)\\
&+ 6t\left(d\phi_{Lip}\|\tilde{f}^{in}\|_{L^\infty}\exp\left( d\phi_M t + d\sigma  W_t\right)\right)^2 \int_0^t  \|\varphi_s - \tilde{\varphi}_s\|_{L^\infty}^2 ds\\
&+ 6t\left(\|\tilde{f}^{in}\|_{L^\infty}\exp\left( d\phi_M t + d\sigma  W_t\right) \right)^2\int_0^t d\phi_M (4\mathcal{R}(s)\mathcal{P}(s))^{2d}\|f_s - \tilde{f}_s\|_{L^\infty}^2 ds\\
&+2\|\tilde{f}_t\|_{W^{1,\infty}}^2\|\varphi_t -\tilde{\varphi}_t\|_{L^\infty}^2.
\end{align*}

\noindent Setting

\begin{align*}
&\mathcal{B}_t^1 := 6\left[1+ t \left(d\|\phi\|_{W^{1,\infty}}\max(\|f^{in}\|_{L^\infty}, \|\tilde{f}^{in}\|_{L^\infty}) \exp(d\phi_Mt + d\sigma W_t)\right)^2\right],\\
&\mathcal{C}_t^1 := (1 + (4\mathcal{R}(t)\mathcal{P}(t))^{2d}),
\end{align*}
we obtain the desired result (i) of \eqref{T3-1.1}.\\

\noindent (ii)~~Now, we estimate $\|\varphi_t -\tilde{\varphi}_t\|_{L^\infty}$. It follows from \eqref{char} and It\^o's lemma that

\begin{align*}
d|V_t - \tilde{V}_t|^2 &= 2(V_t - \tilde{V}_t) d(V_t - \tilde{V}_t) + d(V_t - \tilde{V}_t)d(V_t-\tilde{V}_t)\\
&=2\left( \underbrace{(V_t - \tilde{V}_t) (F_a[f_t](\varphi_t) - F_a[\tilde{f}_t](\tilde{\varphi}_t))}_{=:\mathcal{I}_{23}.} + \sigma^2 |V_t - \tilde{V}_t|^2 \right)dt - 2\sigma |V_t - \tilde{V}_t|^2 dW_t.
\end{align*}
Here, we have

\begin{align*}
\mathcal{I}_{23} &\le \int_{\bbr^{2d}}\left| \phi(x_* - X_t) - \phi(x_* - \tilde{X}_t)\right| |(v_* - V_t)\cdot(V_t - \tilde{V}_t)| f_t  dv_*dx_*\\
&\quad - \int_{\bbr^{2d}}\phi(x_* - \tilde{X}_t)|V_t - \tilde{V}_t|^2 f_t dv_*dx_*\\
&+ \int_{\bbr^{2d}}\phi(x_*-\tilde{X}_t))|(v_* - \tilde{V}_t)\cdot(V_t -\tilde{V}_t)| |f_t - \tilde{f}_t|dv_*dx_*\\
&=:\mathcal{I}_{231} + \mathcal{I}_{232} + \mathcal{I}_{233}. 
\end{align*}

\noindent We separately estimate the $\mathcal{I}_{23i}$'s as follows:

\begin{align*}
&\mathcal{I}_{231} \le 2\phi_{Lip}\mathcal{P}(t)|X_t - \tilde{X}_t| |V_t - \tilde{V}_t|  \le 2\phi_{Lip}\mathcal{P}(t) \|\varphi_t - \tilde{\varphi}_t\|_{L^\infty}^2, \quad \mathcal{I}_{232} \le 0,\\
&\mathcal{I}_{233} \le 2\phi_M\mathcal{P}(t) |V_t - \tilde{V}_t| \|f_t - \tilde{f}_t\|_{L^\infty} (4\mathcal{R}(t) \mathcal{P}(t))^d\\
&\hspace{0.7cm}\le \phi_M \mathcal{P}(t)(4\mathcal{R}(t) \mathcal{P}(t))^d\left(\|f_t - \tilde{f}_t\|_{L^\infty}^2  + \|\varphi_t - \tilde{\varphi}_t\|_{L^\infty}\right)^2.\\
\end{align*}

\noindent Thus, by Lemma \ref{LA-1} we get

\begin{align*}
|V_t - \tilde{V}_t|^2 &\le \phi_M \int_0^t  \mathcal{P}(s)(4\mathcal{R}(s) \mathcal{P}(s))^d\|f_s - \tilde{f}_s\|_{L^\infty}^2 \exp(-2\sigma (W_t - W_s))ds\\
&\quad +(2\phi_{Lip} + \phi_M) \int_0^t \mathcal{P}(s)(4\mathcal{R}(s) \mathcal{P}(s))^d\|\varphi_s - \tilde{\varphi}_s\|_{L^\infty}^2 \exp(-2\sigma (W_t - W_s))ds\\
&\le 2\|\phi\|_{W^{1,\infty}} \exp\left(4\sigma \sup_{0 \le s \le t} |W_s|\right) \\
& \quad \times \int_0^t \mathcal{P}(s)(4\mathcal{R}(s) \mathcal{P}(s))^d(\|\varphi_s - \tilde{\varphi}_s\|_{L^\infty}^2 + \|f_s - \tilde{f}_s\|_{L^\infty}^2 )ds.
\end{align*}

\noindent Moreover, it is easy to obtain that
\[d|X_t - \tilde{X}_t|^2 = 2(X_t - \tilde{X}_t) \cdot(V_t - \tilde{V}_t) \le 2\|\varphi_t - \tilde{\varphi}_t\|_{L^\infty}^2. \]

\noindent Thus, if we define $\mathcal{B}_t^2$ and $\mathcal{C}_t^2$ as
\[
\mathcal{B}_t^2 := 1+2\|\phi\|_{W^{1,\infty}} \exp\left(4\sigma \sup_{0 \le s \le t} |W_s|\right), \quad \mathcal{C}_t^2 := 1 + \mathcal{P}(t)(4\mathcal{R}(t) \mathcal{P}(t))^d,
\]
then (ii) of \eqref{T3-1.1} can be fulfilled with the above $\mathcal{B}_t^2$ and $\mathcal{C}_t^2$.\newline

\noindent Therefore, we add $(i)$ in \eqref{T3-1.1}$_1$ to $(1+2\max(\|f_t\|_{W^{1,\infty}}, \|\tilde{f}_t\|_{W^{1,\infty}}))$ times $(ii)$ in \eqref{T3-1.1}$_2$ and obtain
\begin{align*}
&\|f_t - \tilde{f}_t\|_{L^\infty}^2 + \|\varphi_t - \tilde{\varphi}_t\|_{L^\infty}^2\\
& \hspace{0.2cm} \le \mathcal{B}_t^1\Big[  \|f^{in} - \tilde{f}^{in}\|_{L^\infty}^2 +  \int_0^t  \mathcal{C}_s^1 \Big( \|\varphi_s - \tilde{\varphi}_s\|_{L^\infty}^2+ \|f_s - \tilde{f}_s\|_{L^{\infty}}^2 \Big )ds  \Big] \\
& \hspace{0.5cm} + (1+2\max(\|f_t\|_{W^{1,\infty}}, \|\tilde{f}_t\|_{W^{1,\infty}}) \mathcal{B}_t^2  \left( \ \int_0^t  \mathcal{C}_s^2( \|\varphi_s - \tilde{\varphi}_s\|_{L^\infty}^2+ \|f_s - \tilde{f}_s\|_{L^{\infty}}^2)ds  \right)\\
& \hspace{0.2cm} \le \mathcal{B}_t^1  \|f^{in} - \tilde{f}^{in}\|_{L^\infty}^2 + \tilde{\mathcal{B}}_t \int_0^t  \Big( \|\varphi_s - \tilde{\varphi}_s\|_{L^\infty}^2+ \|f_s - \tilde{f}_s\|_{L^{\infty}}^2\Big)ds,
\end{align*}
where $\tilde{\mathcal{B}}_t$ is given by
\[\tilde{\mathcal{B}}_t := \mathcal{B}_t^1 \left(\sup_{0 \le s \le t}\mathcal{C}_s^1\right) + \left(1+2\max(\|f_t\|_{W^{1,\infty}}, \|\tilde{f}_t\|_{W^{1,\infty}})\right) \mathcal{B}_t^2 \left(\sup_{0 \le s \le t}\mathcal{C}_s^2\right). \]

\noindent Then, letting $y_t :=  \int_0^t  \Big( \|\varphi_s - \tilde{\varphi}_s\|_{L^\infty}^2+ \|f_s - \tilde{f}_s\|_{L^{\infty}}^2\Big)ds  $, we have
\[dy_t \le \left(\mathcal{B}_t^1 \|f^{in} - \tilde{f}^{in}\|_{L^\infty}^2 + \tilde{\mathcal{B}}_t y_t\right)dt. \]
Then, by Gr\"onwall's lemma we get
\[y_t \le \|f^{in} - \tilde{f}^{in}\|_{L^\infty}^2 \int_0^t \mathcal{B}_s^1  \exp\left(\int_s^t \tilde{\mathcal{B}}_\tau d\tau\right) ds, \]
and this gives
\[
\|f_t - \tilde{f}_t\|_{L^\infty}^2 + \|\varphi_t - \tilde{\varphi}_t\|_{L^\infty}^2\le \|f^{in} - \tilde{f}^{in}\|_{L^\infty}^2 \left[\mathcal{B}_t^1 + \tilde{\mathcal{B}}_t \int_0^t \mathcal{B}_s^1  \exp\left(\int_s^t \tilde{\mathcal{B}}_\tau d\tau\right) ds\right].
\]
\noindent Hence, defining 
\[ \mathcal{D}_t := \mathcal{B}_t^1 + \tilde{\mathcal{B}}_t \int_0^t \mathcal{B}_s^1  \exp\left(\int_s^t \tilde{\mathcal{B}}_\tau d\tau\right) ds, \]
we arrive at the desired estimate.

\section{A proof of Proposition \ref{P3.3}}
\setcounter{equation}{0}
Recall that $f_t^{n,\varepsilon}$ satisfies a differential form:
\[
\partial_t f_t^{n,\varepsilon}  = -v \cdot \nabla_x f_t^{n,\varepsilon} - \nabla_v \cdot (F_a[f_t^{n-1,\varepsilon}]f_t^{n,\varepsilon}) + \sigma \nabla_v \cdot (v f_t^{n,\varepsilon}) \circ \dot{W}_t,
\] 
i.e., it satisfies
\begin{equation} \label{Bp-2}
f_t^{n,\varepsilon} = f^{in,\varepsilon} - \int_0^t v \cdot \nabla_x f_s^{n,\varepsilon} ds  - \int_0^t \nabla_v \cdot (F_a[f_t^{n-1,\varepsilon}]f_t^{n,\varepsilon}) ds + \sigma \int_0^t \nabla_v \cdot (vf_t^{n,\varepsilon}) \circ dW_s.
\end{equation}  
Next, we claim: there exists a nonnegative process $\mathcal{A}^m_t$ with continuous sample paths such that
\[
\|f_t\|_{W^{m,\infty}} \le  \|f^{in} \|_{W^{m,\infty}} \mathcal{A}^m_t. 
\]
In the sequel, we provide $L^{\infty}$-estimate of $f_t$ and its derivatives to provide a proof of Proposition \ref{P3.3}.  \newline

\noindent $\bullet$~(Zeroth-order estimate): It follows the formula \eqref{C-2-2} that
\begin{align*}
f_t^{n,\varepsilon}(\varphi_t^{n,\varepsilon}(x,v)) &= f^{in,\varepsilon}(x,v)\exp \left\{ -\int_0^t \nabla_v \cdot F_a[f_s^{n-1,\varepsilon}](\varphi_s^{n,\varepsilon}(x,v))ds + d \sigma W_t \right\}\\
&\le \|f^{in,\varepsilon} \|_{L^\infty}\exp(d\phi_M t +d \sigma  W_t).
\end{align*}
This implies the zeroth-order estiamte:
\begin{equation}\label{Bp-4}
\|f_t^{n,\varepsilon}\|_{L^\infty} \le \|f^{in,\varepsilon}\|_{L^\infty}\exp(d \phi_M t + d \sigma W_t).
\end{equation}

\vspace{0.2cm}

\noindent $\bullet$~(Higher-order estimates):  Let $\alpha$ and $\beta$ be multi-indices satisfying
\[ 1 \le |\alpha|+|\beta| \le m. \]
Then, we apply $\partial_x^{\alpha} \partial_v^{\beta}$ to the relation \eqref{Bp-2} using Theorem 3.1.2 in \cite{H.K}:
\begin{align}
\begin{aligned}\label{Bp-5}
\partial_x^\alpha \partial_v^\beta f_t^{n,\varepsilon} &= \partial_x^\alpha \partial_v^\beta f^{in,\varepsilon} - \sum_{|\mu_1| \le 1}\binom{\beta}{\mu_1} \int_0^t \partial_v^{\mu_1}(v) \cdot \nabla_x (\partial_x^\alpha \partial_v^{\beta-\mu_1} f_s^{n,\varepsilon})ds \\
&\quad - \sum_{\substack{\mu_2 \le \alpha \\ |\mu_3| \le 1}} \binom{\alpha}{\mu_2}\binom{\beta}{\mu_3} \int_0^t \nabla_v \cdot (\partial_x^{\mu_2}\partial_v^{\mu_3} F_a[f_s^{n-1,\varepsilon}] \partial_x^{\alpha-\mu_2}\partial_v^{\beta-\mu_3} f_s^{n,\varepsilon})ds\\
& \quad + \sigma \sum_{|\mu_4|\le 1}\binom{\beta}{\mu_4} \int_0^t \nabla_v \cdot (\partial_v^{\mu_4} (v) \partial_x^{\alpha}\partial_v^{\beta-\mu_4} f_s^{n,\varepsilon}) \circ dW_s,
\end{aligned}
\end{align}
where we used the relation:
\[ \partial_v^{\mu_3} F_a[f_t^{n-1, \varepsilon}]  = 0, \quad \mbox{for}~|\mu_3| \geq 2. \]

Note that the differentiation equality \eqref{Bp-5} is only true outside a $\mathbb{P}$-zero set in $\Omega$ which depends on $(x,v)$, according to Theorem 3.1.2 in \cite{H.K}. However, we can use the argument in Lemma \ref{L2.2} to obtain that the equality also holds $\mathbb{P} \otimes dx \otimes dv$-a.s. Now, we rearrange the previous relation to obtain
\begin{align}
\begin{aligned}\label{App-B.2}
&\partial_x^\alpha\partial_v^\beta f_t^{n,\varepsilon} =\partial_x^\alpha \partial_v^\beta f^{in,\varepsilon} -\int_0^t \Big[ v \cdot \nabla_x (\partial_x^\alpha \partial_v^\beta f_s^{n,\varepsilon}) + F_a[f_s^{n-1,\varepsilon}] \cdot \nabla_v(\partial_x^\alpha \partial_v^\beta f_s^{n,\varepsilon}) \Big] ds \\
&\hspace{0.7cm} +  \sigma \int_0^t v \cdot \nabla_v (\partial_x^\alpha \partial_v^\beta f_s^{n,\varepsilon}) \circ dW_s -\frac{d+|\beta|}{d} \int_0^t \nabla_v \cdot F_a[f_s^{n-1,\varepsilon}] \partial_x^\alpha \partial_v^\beta f_s^{n,\varepsilon}ds \\
& \hspace{0.7cm} + \sigma(d+|\beta|)\int_0^t \partial_x^\alpha \partial_v^\beta f_s^{n,\varepsilon} \circ dW_s -\int_0^t \mathcal{L}_{\alpha,\beta}(s) ds, \quad \mathbb{P}\otimes dx \otimes dv \mbox{-a.s.},
\end{aligned}
\end{align}
where the process $\mathcal{L}_{\alpha,\beta}$ is given by the following relation:
\begin{align*}
\mathcal{L}_{\alpha,\beta} &:= \sum_{|\mu_1| = 1}\binom{\beta}{\mu_1} \partial_v^{\mu_1}(v) \cdot \nabla_x (\partial_x^\alpha \partial_v^{\beta-\mu_1} f_s^{n,\varepsilon}) + \sum_{0 \neq \mu_2 \le \alpha} \binom{\alpha}{\mu_2} \nabla_v \cdot (\partial_x^{\mu_2} F_a[f_s^{n-1,\varepsilon}]) \partial_x^{\alpha-\mu_2}\partial_v^\beta f_s^{n,\varepsilon}\\
&\hspace{0.2cm}+ \sum_{\substack{0 \neq \mu_2\le \alpha \\ |\mu_3| = 1 }}\binom{\alpha}{\mu_2}\binom{\beta}{\mu_3}\partial_x^{\mu_2} \partial_v^{\mu_3} F_a[f_s^{n-1,\varepsilon}]\cdot \nabla_v (\partial_x^{\alpha-\mu_2}\partial_v^{\beta-\mu_3} f_s^{n,\varepsilon})\\
&\hspace{0.2cm} +  \sum_{0 \neq \mu_2 \le \alpha} \binom{\alpha}{\mu_2}\partial_x^{\mu_2} F_a[f_s^{n-1,\varepsilon}] \cdot \nabla_v( \partial_x^{\alpha-\mu_2}\partial_v^\beta f_s^{n,\varepsilon}).
\end{align*}

\noindent Next, we define $\lambda$ and $\tilde\lambda$ as follows:
\begin{align*}
\lambda_t(x,v) &:= \partial_x^\alpha\partial_v^\beta f^{in,\varepsilon}(x,v) -\frac{d+|\beta|}{d} \int_0^t \lambda_s(x,v)(\nabla_v \cdot F_a[f_s^{n-1,\varepsilon}])(\varphi_s^{n,\varepsilon})ds \\
&\quad + \sigma(d+|\beta|)\int_0^t\lambda_s(x,v)  \circ dW_s -\int_0^t \mathcal{L}_{\alpha,\beta}(\varphi_s^{n,\varepsilon})ds,\\
\tilde\lambda_t(x,v) &:= \lambda_t ((\varphi_t^{n,\varepsilon})^{-1}).
\end{align*}
\noindent By using generalized It\^o's formula from Theorem 3.3.2 in \cite{H.K}, $\tilde\lambda_t$ satisfies the relation \eqref{App-B.2}. Thus, by the uniqueness, 
\[  \tilde\lambda_t = \partial_x^\alpha\partial_v^\beta f_t^{n,\varepsilon}, \]
and we use It\^o's formula on $\lambda_t$ to get
\begin{align*}
&\partial_x^\alpha \partial_v^\beta f_t^{n,\varepsilon}(\varphi_t^{n,\varepsilon}) =\partial_x^\alpha\partial_v^\beta f^{in,\varepsilon}(x,v)\exp\Big[  -\frac{d+|\beta|}{d}\int_0^t \nabla_v \cdot F_a[f_s^{n-1,\varepsilon}](\varphi_s^{n,\varepsilon})ds + \sigma(d+|\beta|)W_t  \Big] \\
&\hspace{1cm} - \int_0^t \exp\Big[ -\frac{d+|\beta|}{d}\int_s^t \nabla_v \cdot F_a[f_{\tau}^{n-1,\varepsilon}](\varphi_\tau^{n,\varepsilon})d\tau + \sigma(d+|\beta|)(W_t-W_s)  \Big]  \\
& \hspace{1.2cm} \times \mathcal{L}_{\alpha,\beta}(s,\varphi_s^{n,\varepsilon})ds.
\end{align*}
For detailed explanation for the above realtion, we refer to the proof of Theorem 3.2 in \cite{Chow}. \newline

\noindent Note that the following estimates hold: 
\begin{itemize}
\item
If $|\beta|=1$, one has
\[|\partial_x^\alpha \partial_v^\beta F_a[f^{n-1,\varepsilon}_t]| \le \|\phi\|_{\mathcal{C}^m}. \]
\item
If $|\alpha| \ge 1$, one gets
\begin{align*}
|\partial_x^\alpha F_a[f^{n-1,\varepsilon}_t](\varphi_t^{n,\varepsilon})| &\le \|\phi\|_{\mathcal{C}^m} \int_{\bbr^{2d}} |v_* \cdot V_t^{n,\varepsilon}| f^{n-1,\varepsilon}_t(x_*,v_*)dv_*dx_* \\
&\le \|\phi\|_{\mathcal{C}^m} (\mathcal{V}^{n-1,\varepsilon}(t)) |V_t^{n,\varepsilon}|\le \|\phi\|_{\mathcal{C}^m} (\mathcal{V}^\infty(t))^2.
\end{align*}
\end{itemize}
We set $C_{\alpha,\beta}(t)$ to be
\[
C_{\alpha,\beta}(t) := \|\phi\|_{\mathcal{C}^m} \left(\sum_{|\mu_1| = 1} \binom{\beta}{\mu_1} + \sum_{0 \ne \mu_2 \le \alpha} \binom{\alpha}{\mu_2} + \sum_{\substack{0\le \mu_2 \le \alpha \\ |\mu_3| =1}}\binom{\alpha}{\mu_2}\binom{\beta}{\mu_3}\right)(1+(\mathcal{V}^\infty(t))^2).
\]
This yields
\[|\mathcal{L}_{\alpha,\beta}(t,\varphi_t^{n,\varepsilon})| \le C_{\alpha,\beta}(t) \|f_t^{n,\varepsilon}\|_{W^{m,\infty}}. \]
Thus, we have
\begin{align}
\begin{aligned}\label{T2-1.3}
\partial_x^\alpha \partial_v^\beta f_t^{n,\varepsilon}(\varphi_t^{n,\varepsilon}) &\le \|\partial_x^\alpha\partial_v^\beta f^{in,\varepsilon}\|_{L^\infty} \exp((d+|\beta|)(\phi_M t + \sigma W_t)\\
&+ \int_0^t \exp((d+|\beta|)\{\phi_M (t-s) + \sigma (W_t-W_s)\}C_{\alpha,\beta}(s) \|f^{n,\varepsilon}_s\|_{W^{m,\infty}}ds.
\end{aligned}
\end{align}
Now, we take the supremum over all characteristic flow, sum \eqref{T2-1.3} over all $1 \le |\alpha|+|\beta| \le m$ and combine this with \eqref{Bp-4} to obtain
\begin{align*}
\begin{aligned}
\|f^{n,\varepsilon}_t\|_{W^{m,\infty}} &\le \|f^{in,\varepsilon}\|_{W^{m,\infty}} \mathcal{M}_t^m  + \mathcal{M}_t^m \int_0^t \Big[ \exp(-(d+m)\phi_M s) \\
&\hspace{0.2cm} \times\sum_{|\alpha|+|\beta| \le m}\exp(-\sigma(d+|\beta|)W_s)C_{\alpha,\beta}(s) \|f^{n,\varepsilon}_s\|_{W^{m,\infty}} \Big] ds,
\end{aligned}
\end{align*}
where the process $\mathcal{M}_t^m$ is given by the following relation:
\[\mathcal{M}_t^m :=  \exp((d+m)\phi_M t) \sum_{|\beta| \le m} \exp(\sigma (d+|\beta|)W_t).\]
Note that $\mathcal{M}_t^m$ is independent of $n$ and $\varepsilon$. We set 
\[ b_n(t) := \|f^{n,\varepsilon}_t\|_{W^{m,\infty}}(\mathcal{M}_t^m)^{-1}. \]
Then, one gets
\[
b_{n+1}(t) \le b_0 + \int_0^t \tilde{\mathcal{N}}_s^m b_{n+1}(s)ds,
\]
where the process $ \tilde{\mathcal{N}}_s^m$ is 
\[ \tilde{\mathcal{N}}_s^m := \left\{\sum_{|\beta|\le m}\exp(\sigma(N+|\beta|)W_s) \right\}\left\{\sum_{|\beta| \le m}\exp(-\sigma(N+|\beta|)W_s)\right\} \left( \sum_{|\alpha|+|\beta|\le m}C_{\alpha,\beta}(s)\right).\]
\noindent Thus, we can use Gr\"onwall's lemma to obtain
\[ \|f^{n,\varepsilon}_t\|_{W^{m,\infty}} \le \|f^{in,\varepsilon}\|_{W^{m,\infty}}\mathcal{A}^m_t, \]
where the process ${\mathcal A}_t^m$ is given by the following relation:
\begin{align*}
&\mathcal{A}_t^m := \exp((d+m)\phi_M t) \sum_{|\beta| \le m} \exp(\sigma (d+|\beta|)W_t)\\
& \times\hspace{-0.05cm} \exp\hspace{-0.05cm}\left[\int_0^t \hspace{-0.05cm}\left\{\hspace{-0.05cm}\sum_{|\beta|\le m}\hspace{-0.1cm}\exp(\sigma(d+|\beta|)W_s) \hspace{-0.1cm}\right\}\hspace{-0.15cm}\left\{\hspace{-0.05cm}\sum_{|\beta| \le m}\hspace{-0.1cm}\exp(-\sigma(d+|\beta|)W_s)\hspace{-0.1cm}\right\} \hspace{-0.1cm}\left( \sum_{|\alpha|+|\beta|\le m}\hspace{-0.2cm}C_{\alpha,\beta}(s)\hspace{-0.05cm}\right)\hspace{-0.05cm}ds \right].
\end{align*}


\newpage



\begin{thebibliography}{99}


\bibitem{A-H} Ahn, S. and Ha, S.-Y.: \textit{Stochastic flocking dynamics of the Cucker-Smale model with multiplicative white noises.} J. Math. Phys. {\bf 51} (2010), 103301.

\bibitem{A-P-Z} Albi, G., Pareschi, L. and Zanella, M.: \textit{Uncertain quantification in control problems for flocking models.}
Math. Probl. Eng.  Art. ID (2015), 850124.


\bibitem{B-D-G-M}Boudin, L., Desvillettes, L., Grandmont, C. and Moussa, A.: \textit{Global existence of solutions for the coupled Vlasov and Navier-Stokes equations.} Differential and integral equations, {\bf 22} (2009), 1247-1271.


\bibitem{C-F-R-T} Carrillo, J. A., Fornasier, M., Rosado, J. and Toscani, G.: \textit{Asymptotic flocking dynamics for the kinetic Cucker-Smale model}. SIAM J. Math. Anal. {\bf 42} (2010), 218-236.

\bibitem{C-F-T-V} Carrillo, J. A. Fornasier, M., Toscani, G. and Vecil,
F.: \textit{Particle, kinetic, and hydrodynamic models of swarming.
Mathematical modeling of collective behavior in socio-economic and
life sciences.} 297-336, Model. Simul. Sci. Eng. Technol.,
Birkhauser Boston, Inc., Boston, MA, 2010.

\bibitem{C-P-Z} Carrillo, J. A., Pareschi, L. and Zanella, M.: \textit{Particle based gPC methods for mean-field models of swarming with uncertainty}. To appear in Comm. in Comp. Phys.



\bibitem{C-H-L} Choi, Y.-P., Ha, S.-Y. and Li, Z.:  \textit{Emergent dynamics of the Cucker-Smale flocking model and its variants.}  In N. Bellomo, P. Degond, and E. Tadmor (Eds.), Active Particles Vol.I - Theory, Models, Applications(tentative title), Series: Modeling and Simulation in Science and Technology, Birkhauser-Springer.

\bibitem{C-S2} Choi, Y.-P., Salem, S.: \textit{Cucker-Smale flocking particles with multiplicative noises: Stochastic mean-field limit and phase transition}. Kinet. Relat. Models {\bf 12} (2019), 573-592.


\bibitem{Chow} Chow, P.-L.: \textit{Stochastic Partial Differential Equations}, Chapman and Hall/CRC, 2015.

\bibitem{C-F} Coghi, M. and Flandoli, F.: \textit{Propagation of chaos for interacting particles subject to environmental noise}. Ann. Appl. Probab. {\bf 26} (2016), 1407-1442.


\bibitem{C-D1} Cucker, F. and Dong, J.-G.: \textit{On flocks influenced by closest neighbors.}  Math. Models Methods Appl. Sci. {\bf 26} (2016), 2685-2708.

\bibitem{C-D2}  Cucker, F. and Dong, J.-G.: \textit{A general collision-avoiding flocking framework.} IEEE Trans. Automat. Control {\bf 56} (2011), 1124-1129.

\bibitem{C-M} Cucker, F. and Mordecki, E.: \textit{Flocking in noisy environments.} J. Math. Pures Appl. {\bf 89} (2008), 278-296.

\bibitem{C-S} Cucker, F. and Smale, S.: \textit{Emergent behavior in flocks}. IEEE Trans. Automat. Control {\bf 52} (2007), 852-862.

\bibitem{D-M} Degond, P. and Motsch, S.: \textit{Large-scale dynamics of the persistent turing walker model of fish behavior.} J. Stat. Phys., {\bf 131} (2008), 989-1022.



\bibitem{E-H-S} Erban, R., Haskovec, J. and Sun, Y.: \textit{A Cucker--Smale model with noise and delay.} SIAM J. Appl. Math., {\bf 76} (2016), 1535-1557.

\bibitem{Ev} Evans, L. C.: \textit{An introduction to stochastic differential equations.} American Mathematical Soc., 2012.

\bibitem{F-G-P}Flandoli, F., Gubinelli, M. and Priola, E.: \textit{Well-posedness of the transport equation by stochastic perturbation.} Invent. Math., {\bf 180} (2010), 1-53.


\bibitem{H-J} Ha, S.-Y. and Jin, S.: \textit{Local sensitivity analysis for the Cucker-Smale model with random inputs.} Kinetic Relat. Models {\bf 11} (2018), 859-889.

\bibitem{H-J-J}  Ha, S.-Y., Jin, S. and Jung, J.: \textit{A local sensitivity analysis for the kinetic Cucker-Smale equation with random inputs.} J. Differential Equations {\bf 265} (2018), 3618-3649. 

\bibitem{H-J-N-X-Z} Ha, S.-Y., Jeong, J., Noh, S. E., Xiao, Q. and Zhang, X.: \textit{Emergent dynamics of Cucker-Smale flocking particles in a random environment.} J. Differential Equations {\bf 262} (2017), 2554-2591.

\bibitem{H-K-Z} Ha, S.-Y., Kim, J. and Zhang, X.: \textit{Uniform stability of the Cucker-Smale model and its application to the mean-field limit.} Kinetic Relat. Models {\bf 11} (2018), 1157-1181.

\bibitem{H-L-L} Ha, S.-Y., Lee, K. and Levy, D.: \textit{Emergence of time-asymptotic flocking in a stochastic Cucker-Smale system.}
Commun. Math. Sci. {\bf 7} (2009), 453-469.

\bibitem{H-Liu} Ha, S.-Y. and Liu, J.-G.: \textit{A simple proof of Cucker-Smale flocking dynamics and mean field limit}.
    Commun. Math. Sci. {\bf 7} (2009), 297-325.

\bibitem{H-T} Ha, S.-Y. and Tadmor, E.:
    \textit{From particle to kinetic and hydrodynamic description of flocking}.
    Kinetic Relat. Models {\bf 1} (2008), 415-435.











\bibitem{H.K} Kunita, H.: \textit{Stochastic flows and stochastic differential equations.} Cambridge University Press, Cambridge, 1990.


\bibitem{L-P-L-S} Leonard, N. E., Paley, D. A., Lekien, F., Sepulchre, R., Fratantoni, D. M. and Davis, R. E.:\textit{Collective motion, sensor networks and ocean sampling.} Proc. IEEE {\bf 95} (2007), 48-74.

\bibitem{L-W} Li, Q. and Wang, L.: \textit{Uniform regularity for linear kinetic equations with random input based on hypocoercivity.}
  SIAM/ASA J. Uncertainty Quantification {\bf 5} (2017), 1193-1219.


\bibitem{M-T1} Motsch, S. and Tadmor, E.: \textit{Heterophilious dynamics: enhanced consensus.} SIAM Review {\bf 56} (2014), 577-621.

\bibitem{M-T2} Motsch, S. and Tadmor, E.: \textit{A new model for self-organized dynamics and its flocking behavior.}
J. Stat. Phys. {\bf 144} (2011), 923-947.




\bibitem{P-L-S-G-P} Paley, D. A., Leonard, N. E., Sepulchre, R., Grunbaum, D. and Parrish, J. K.: \textit{Oscillator models and
collective motion.} IEEE Control Systems Magazine {\bf 27} (2007), 89-105.

\bibitem{P-E-G} Perea, L., Elosegui, P. and G\'{o}mez, G.: \textit{Extension of the Cucker-Smale control law to space flight formation.}
J. of Guidance, Control and Dynamics {\bf 32} (2009), 527-537.

\bibitem{P-S} Punshon-Smith, S. and Smith, S.: \textit{On the Boltzmann equation with stochastic kinetic transport: global existence of renormalized martingale solutions.} Arch. Ration. Mech. Anal. {\bf 229} (2018), 627-708.


\bibitem{Re} Reynolds, C. W.: \textit{Flocks, Herds, and Schools:
A Distributed Behavioral Model.} Computer Graphics, 21(4), July 1987, pp. 25-34.
(ACM SIGGRAPH '87 Conference Proceedings, Anaheim, California, July 1987.


\bibitem{Ro} Rosello, A.: \textit{Weak and strong mean-field limits for stochastic Cucker-Smale particle systems}. Preprint. arXiv:1905.02499.

\bibitem{S-P} Schilling, R. L. and Partzsch, L.: \textit{Brownian motion: an introduction to stochastic processes.} Walter de Gruyter GmbH \& Co KG, 2014.

\bibitem{S-V}Stroock, D. and Varadhan, S. R. S.: \textit{On the support of diffusion processes with applications to the strong maximum principle}, Proc. Sixth Berkeley Symp. on Math. Statist. and Prob., {\bf 3} (1972), 333–359.




  
\bibitem{Tad}  Tadmor, E.: \textit{Mathematical aspects of self-organized dynamics: consensus, emergence of leaders, and social hydrodynamics.} SIAM News, {\bf 48}, 2015.

\bibitem{T-T} Toner, J. and Tu, Y.: \textit{Flocks, herds, and Schools: A quantitative theory of flocking.} Physical Review E. {\bf 58} (1998), 4828-4858.

\bibitem{V} Veraar, M.: \textit{The stochastic Fubini theorem revisited}, Stochastics, {\bf 84}, (2012), 543-551.


\bibitem{V-Z} Vicsek, T and Zefeiris, A.: \textit{Collective motion.}  Phys. Rep. {\bf 517} (2012), 71-140.

\bibitem{V-C-B-C-S} Vicsek, T., Czir\'{o}k, E. Ben-Jacob, I. Cohen and O. Schochet: \textit{Novel type of phase transition in a system of self-driven particles.} Phys. Rev. Lett. {\bf 75} (1995), 1226-1229.

\bibitem{W-Z} Wong, E. and Zakai, M.: \textit{On the convergence of ordinary integrals to stochastic integrals}. Ann. Math. Statist. {\bf 36} (1965), 1560-1564.

\bibitem{W-Z2} Wong, E. and Zakai, M.: \textit{On the relation between ordinary and stochastic differential equations}. Internat. J. Engrg. Sci. {\bf 3} (1965), 213-229.

\end{thebibliography}
\end{document}